\documentclass[11pt,reqno]{article}
\usepackage[margin=0.96in, top=2.2cm,bottom=2.2cm]{geometry}

\usepackage[english]{babel}
\usepackage[T1]{fontenc}
\usepackage{setspace}
\usepackage{enumitem}
\usepackage[normalem]{ulem}
\usepackage{float}
\usepackage{scrextend}
\deffootnote{0em}{1.6em}{\thefootnotemark.\enskip}
\usepackage{xcolor}

\usepackage{amsmath}
\usepackage[makeroom]{cancel}
\usepackage{amssymb}
\usepackage{amsthm}
\usepackage{mathtools}
\usepackage{euscript,mathrsfs}
\usepackage{bbm}
\PassOptionsToPackage{hyphens}{url}
\usepackage[colorlinks=true, allcolors=blue]{hyperref}
\usepackage{graphicx}
\usepackage[title]{appendix}
\usepackage{bm}
\usepackage{subcaption}

\mathtoolsset{showonlyrefs}

\hypersetup{
    colorlinks,
    linkcolor={blue},
    citecolor={blue},
    urlcolor={blue}
}
\usepackage{yfonts}
\usepackage{indentfirst}
\usepackage[
    citestyle=numeric,
    bibstyle=authoryear,
    dashed=false,
    maxnames=4]{biblatex}
\makeatletter
\input{numeric.bbx}
\makeatother
\renewbibmacro{in:}{}
\DeclareNameAlias{author}{family-given}

\DeclareFieldFormat*{title}{#1}
\DeclareFieldFormat*[book]{title}{\textit{#1}}
\DeclareFieldFormat*{volume}{\textbf{#1}}

\newtheorem{Th}{Theorem}[section]

\newtheorem{Lem}[Th]{Lemma}
\newtheorem{Prop}[Th]{Proposition}
\newtheorem{Cor}[Th]{Corollary}
\theoremstyle{definition}
\newtheorem{remark}[Th]{Remark}

\numberwithin{equation}{section}

\newcommand{\fc}{\mathcal{F}}
\newcommand{\pr}{\mathbb{P}}
\newcommand{\ex}{\mathbb{E}}
\newcommand{\T}{\mathcal{T}}

\title{\textsc{Optimal Stopping for a Diffusion with Unobserved Bernoulli Drift}}
\author{
\textsc{Georgy Gaitsgori} 
\thanks{ 
\,\textsc{Columbia University, Department of Mathematics, 2990 Broadway, New York, NY 10027, USA} (e-mail: {\it gg2793@columbia.edu})}
\and
\textsc{Ioannis Karatzas}
\thanks{ 
\,\textsc{Columbia University, Department of Mathematics, 2990 Broadway, New York, NY 10027, USA} (e-mail: {\it ik1@columbia.edu})}
}
\date{\today}

\begin{document}
\maketitle

\begin{abstract}
    We solve fairly explicitly an optimal stopping problem for a Wiener process with unobserved Bernoulli drift, in the presence of a cost on terminal position which is symmetric and increases with distance from the origin, and of a fixed positive cost per unit time \(c > 0\).
    After filtering, the problem reduces to Markovian optimal stopping with complete observations for the state process ``centered'' by its starting position $x \in \mathbb R$.
    However, the solution becomes possible only after foliating by an additional state-parameter \(y \in \mathbb{R}\), representing the displacement from the initial position; this foliation ``lifts'' the problem from the real line to the plane, solves the augmented problem for each fixed initial position \(x\), characterizes fairly explicitly the optimal stopping region in \((x,y)\)-space, and finally obtains the solution of the original problem by ``slicing'' along \(y=0\).
    Following this procedure, we show that, under suitable structural assumptions on the terminal cost, each fixed-\(x\) continuation section is either empty or a single bounded interval, whose endpoints are determined uniquely by a balancing condition; the corresponding value function is then given in semi-explicit form.  The two-dimensional continuation region is obtained by gluing these fixed-\(x\) intervals over \(x\); its two free boundaries satisfy natural monotonicity properties and, at regular points, can be described by a coupled system of ordinary differential equations.  The resulting description yields a threshold-type solution of the original one-dimensional problem whenever the horizontal slice \(y=0\) enters the two-dimensional continuation region.
\end{abstract}



\section{Introduction}\label{sec:intro}

We formulate and solve a basic problem in what might be called ``adaptive sequential analysis'': The optimal stopping of the state process   $X(t) = x + B t + W(t), ~ 0 \le t < \infty$ with initial position $x \in \mathbb{R}$ and unobserved  Bernoulli drift $B$, independent of the driving Brownian motion $W(\cdot)$,  in the presence of a fixed cost per unit of elapsed time $c > 0$, and of a terminal cost $k(\cdot)$ which increases with distance from the origin. In other words, the objective of our problem is to minimize
\[
\ex \big[k\big(X(\tau)\big) + c\tau \big]
\]
over all $X$--stopping times $\tau$.
Basic filtering theory reduces this problem, via the familiar innovations process, to Markovian optimal stopping with complete observations for the state process ``centered'' by its starting position; to wit, for $Y(t) = X(t) -x, ~ 0 \le t< \infty$. 

Even after this reduction, the resulting problem is quite hard to solve directly. This becomes possible only after an additional parametrization, now by the initial position of the so-centered process. Starting $Y(\cdot)$ at an arbitrary position $y \in \mathbb R$ allows the auxiliary problem to be solved fairly explicitly for each fixed value of the original starting position $x = X(0)$. We are thus led to optimal continuation and stopping regions for the two-dimensional Markovian problem; ``free boundaries'' in $(x,y)-$space demarcate these regions, and are described, at regular points, by a system of coupled differential equations. With these regions so-constructed, we evaluate at $y=0$ and obtain the optimal continuation and stopping intervals for the original problem in the $x-$variable. 

The paper should also be viewed as a step in a broader program. Our interest in it grew out of the recent work \cite{BGK25}, which treats a related control--plus--stopping problem for a special class of one-dimensional diffusions, but with fully observed drift. The problem studied here isolates the complementary filtering--plus--stopping component. The longer-term objective is to combine these two directions and understand a triple problem in which the decision maker controls the state dynamics, learns an unobserved parameter through partial observations, and chooses an optimal stopping time.
From this point of view, the Bernoulli-drift model considered below is deliberately simple: it is rich enough to display a nontrivial interaction between filtering and stopping, while still allowing for an explicit geometric analysis that may serve as a building block for more general models.

\smallskip
\noindent\emph{Related Work.} Optimal stopping problems under incomplete information arise in at least two rather different forms. In the first, and more statistical, the unknown quantity is itself part of the objective. This is the case in sequential testing, quickest detection, disorder problems, and Bayesian sequential estimation: one stops in order to decide between competing hypotheses, detect a change, or estimate an unobserved parameter, and the terminal loss is formulated directly in terms of this hidden quantity. In the context of Brownian, or more generally diffusive observations, the classical continuous-time sequential testing goes back to Shiryaev \cite{Shiryaev67}; later developments, among others, include finite-horizon and general diffusion generalizations \cite{GapPes04,GapShi11}, Bayesian drift testing and estimation problems \cite{EksKarVai22, EksVai15}, stochastic-deadline and soft-classification  formulations \cite{CGGK25, CampZha24}, and several quickest-detection or disorder-type problems, including \cite{de2022quickest,MilEks24,ErnPes24,ErnPesZho20, Peskir12, PesShi00,Shiryaev63}. In all these settings, learning is tied directly to the criterion being optimized: the posterior distribution enters not only the dynamics of the sufficient statistic, but also the terminal decision or loss. 

The present problem belongs to a different class. Here, the hidden variable is a feature of the system rather than the target of the terminal decision: it changes the law under which the observed state evolves, and hence the value of waiting, but is not itself classified, estimated, or penalized in the objective. In our model, the decision maker learns about \(B\), because \(B\) determines the drift of the observed process; the criterion, however, depends only on elapsed time and on the observed terminal position. Compared with the line of work just mentioned, results of this structural type, especially explicit, seem far less numerous, and most of the closest examples come from financial applications. Related works include, among others, unknown-state stopping problems \cite{EkstromWang2024}, optimal stopping of the sample mean of a Wiener process with unknown drift \cite{SimonsYao1989}, investment timing under incomplete information \cite{DecampsMariottiVilleneuve2005}, optimal liquidation and selling problems with unknown drift \cite{EkstromLu2011,EkstromVaicenavicius2016}, American option problems under partial information \cite{EkstromVannestal2019,Gapeev2012,GapeevAlMotairi2018}, hidden-Markov-model formulations \cite{Gapeev2022}, and stopping problems for diffusions with unknown structural parameters, such as Brownian bridges with unknown pinning point or pinning time \cite{EkstromVaicenavicius2020,Glover2022}. 
The works \cite{DecampsMariottiVilleneuve2005,EkstromLu2011,EkstromVaicenavicius2016,EkstromVannestal2019} are particularly close to our setting. However, unlike our setup, \cite{DecampsMariottiVilleneuve2005,EkstromLu2011,EkstromVannestal2019} deal with geometric Brownian motion, \cite{EkstromLu2011,EkstromVaicenavicius2016,EkstromVannestal2019} consider exponentially discounted problems, and all these works consider maximization problems of linear or the so-called ``hockey-stick'' terminal costs.

Finally, we mention two adjacent directions in which optimal stopping under incomplete information is coupled with an additional feature. The first adds strategic interaction leading to optimal stopping games. Such games, under incomplete or asymmetric information, have attracted considerable attention in recent years. In such problems, uncertainty may come from the environment itself, from asymmetric information or heterogeneous beliefs among the players, or from different players observing different signals. We refer to \cite{DeAngelisEkstrom2020,DeAngelisEkstromGlover2022,DeAngelisGensbittelVilleneuve2021,DeAngelisMerkulovPalczewski2022,GaitsgoriGroenewald2025} and their references.
The second adjacent direction adds control, leading to what we call control--filtering--stopping, or ``triple'', problems.
General formulations of optimal stopping for controlled diffusions under partial observation go back at least to \cite{BenLio82}, as well as to \cite{FujitaNisio1986}. More explicit Brownian learning models with controlled observation or signal acquisition include \cite{BayKra15,CGGK25,DalShi,DaySez16,de2022quickest,EksKar,MilEks24,HarSun}; these are close to the broader program that motivates the present work. However, in all ``triple'' models known to us, the hidden variable remains part of the objective. The future problem we have in mind is different in this respect: the hidden variable will be a structural parameter of the system, and the terminal criterion a cost of the observed state. Here, we treat the corresponding filtering--plus--stopping problem before the control component is introduced; together with the control--plus--stopping analysis of \cite{BGK25}, this provides the two basic building blocks for a later control--plus--filtering--plus--stopping formulation.

\smallskip
\noindent\emph{Preview.} The paper is relatively long and the solution passes through several logical stages, so we give here a brief map of the argument. The procedure has four main steps. First, we filter out the unobserved drift and recenter the observations around the displacement from the initial position. Secondly, we embed the resulting problem into a two-parameter Markovian family \(U(x,y)\), where \(y\) is the displacement variable and \(x\) represents the original starting position. Thirdly, for each fixed \(x\), we solve the resulting one-dimensional optimal stopping problem in the variable \(y\). Finally, we glue these fixed-\(x\) solutions together and study the continuation region in the full \((x,y)\)-plane, and return to the original one-dimensional problem by taking the horizontal slice \(y=0\).

The passage through the \((x,y)\)-plane is one of the main organizing ideas of the paper. Enlarged state spaces are common in optimal stopping under incomplete information; see, for instance, \cite{DecampsMariottiVilleneuve2005,EkstromLu2011,EkstromVaicenavicius2016,EkstromVannestal2019,EkstromWang2024,Gapeev2012}. However, the role of the second coordinate in the present work is quite different. In many filtered stopping problems (in particular, in all mentioned in the previous sentence), the additional coordinate is itself a genuinely dynamic posterior belief, conditional mean, or other sufficient statistic. In the present model, after centering, the embedded stopping problem has only one dynamic coordinate, namely the displacement variable \(y\), while the original starting point \(x\) enters as a parameter through the terminal cost. Therefore, the \((x,y)\)-plane should be viewed less as the state space of a genuinely two-dimensional diffusion, and more as a state-parameter plane. Thus, the solution is obtained by rotating the point of view: we solve vertical one-dimensional sections, glue them into a two-dimensional geometric object, then recover the original problem by looking at a horizontal section. This lift-and-slice methodology, together with the geometric structure it reveals, is one of our main contributions here.

\smallskip
\noindent\emph{Organization.} 
Section~\ref{sec:model} formulates the model and states our main assumptions. Section~\ref{sec:filtering_reduction} performs the filtering reduction, introduces the centered process, and embeds the original problem into the two-parameter family \(U(x,y)\). Section~\ref{sec:one_d_problem} solves the fixed-\(x\) one-dimensional problem in the displacement variable and develops the role of \(\Phi_x\), the one-valley geometry, and the balancing condition. Section~\ref{sec:structure} glues the fixed-\(x\) solutions into a two-dimensional continuation strip and studies its free boundaries, expansion properties, regularity, diagonal contact, and asymptotics. Section~\ref{sec:original_problem} recovers the original stopping problem by restricting to the horizontal line \(y=0\). Section~\ref{sec:examples} presents examples and illustrations of the different regimes covered by the theory, as well as phenomena that may occur outside the structural assumptions; Section \ref{sec:conclusion} concludes by stating several open problems and directions.

\section{Model}\label{sec:model}

We fix a probability space $(\Omega, \fc, \pr)$, supporting a Bernoulli random variable $B$ with $\pr(B=1) = p = 1 - \pr(B = -1)$ for some $ 0 < p < 1$, and an independent standard Brownian motion $W = \{W(t), \, 0 \le t < \infty\}$. Neither \(B\) nor \(W\) is observed directly.
Instead, we observe the \textit{state process}
\begin{equation}\label{diffusion_process_def}
    X(t)=x+Bt+W(t),\qquad 0\le t<\infty,
\end{equation}
where \(x\in\mathbb R\) is a given initial position, and denote by \(\mathbb F^X=\{\fc^X(t),\,0\le t<\infty\}\) the filtration it generates. We also denote by \(\T^X\) the collection of all \(\mathbb F^X\)-stopping times.

For a given initial position \(x\in\mathbb R\), the objective is to minimize
\begin{equation}\label{objective_function_def}
    J(x,\tau)
    \coloneqq
    \ex\Big[k(X(\tau))+c\,\tau\Big]
\end{equation}
over all \(\tau\in\T^X\), with the convention $J(x, \tau) = \infty$ whenever $\ex[\tau] = \infty$. Here, \(c>0\) is a ``running'' cost per unit of time, and \(k(\cdot)\) is a ``terminal'' cost.
Throughout the paper, we impose the following standing assumption:

\begin{enumerate}[label=\textbf{(A\arabic*)}, leftmargin=4em]
    \item The function $k: \mathbb{R} \to [0, \infty)$ is even, convex, and belongs to
    \(C^2((0,\infty))\). \label{ass_1}
\end{enumerate}

The evenness of \(k(\cdot)\) means that the terminal cost depends only on the distance from the origin, and is assumed without loss of generality. The assumption \(C^2((0,\infty))\) deliberately does not impose differentiability at zero. This is important for costs such as \(|x|\) and \(e^{\lambda |x|}\), and allows the right derivative
\(
    k'(0+)\coloneqq \lim_{s\downarrow0}k'(s)
\)
to be positive. 
This quantity will play a vital role in the diagonal-contact criterion below. As we shall see in Section \ref{sec:structure}, the nature of the solution to this problem depends crucially on how this ``marginal cost'' $k'(0+)$ compares to the continuation cost $c > 0$.

For several structural results we shall impose one or both of the following assumptions.

\begin{enumerate}[label=\textbf{(A\arabic*)}, leftmargin=4em, itemsep=-10pt]
\setcounter{enumi}{1}
    \item
    \(k'(x)>0\) for all \(x>0\), and the function $P_k$ below is convex on \((0,\infty)\):
    \begin{equation}\label{Qk_def}
        P_k(x)\coloneqq \frac{k''(x)+2c}{2k'(x)}, \qquad x \in (0, \infty).
    \end{equation}
    \label{ass_2}

    \item
    \(k'(x)>0\) for all \(x>0\), and \(k'\) is log-concave on
    \((0,\infty)\); equivalently, the function
    \(k''/k'\)
    is non-increasing on \((0,\infty)\). \label{ass_3}
\end{enumerate}

We emphasize that assumptions \ref{ass_2}--\ref{ass_3} are not intended to be sharp. Rather, they are convenient and easy-to-check sufficient conditions, under which all of the structural properties established in this paper hold. In several places below, we shall work under weaker assumptions, which are closer to the true mechanism behind the corresponding results; however, these weaker hypotheses are more technical and most naturally formulated in terms of auxiliary objects introduced only later. We note that all our assumptions \ref{ass_1}--\ref{ass_3} are satisfied by the power and exponential costs
\[
    k(x) = \alpha |x|^\beta + \gamma, \quad \alpha > 0, \ \beta \geq 1, \ \gamma \ge 0;
\qquad \qquad 
    k(x) = \alpha\big(e^{\lambda |x|} - \gamma \big), \quad \alpha,\lambda > 0, \ \gamma \le 1.
\]
These will be the main examples considered throughout the paper.

We are now ready to study the problem \eqref{objective_function_def}, and denote its value function by
\begin{equation}\label{value_function_def}
    V(x) \coloneqq \inf_{\tau \in \T^X} J(x, \tau), \qquad x \in \mathbb R.
\end{equation}

\section{Filtering and Two-Dimensional Markovian Embedding}\label{sec:filtering_reduction}

As is typical in problems with incomplete information, we start by reducing the partially observed problem to a completely observed one. To do so, we introduce
\begin{equation}\label{conditional_exp_for_B}
    \widehat B(t)\coloneqq \ex\left[B\mid \fc^X(t)\right],
    \qquad 0\le t<\infty,
\end{equation}
the conditional mean of the unobserved drift (commonly referred to below as ``posterior drift''). By standard filtering theory (see, e.g., Chapter 3 of Bain and Crisan \cite{BainCrisan}), the observation process admits the $\mathbb F^X$--semimartingale representation
\begin{equation}\label{filtered_sde}
    X(t)=x+\int_0^t \widehat B(s)\,ds+N(t),
    \qquad 0\le t<\infty,
\end{equation}
where $N(\cdot)$ is the so-called \textit{innovations process}, which is a standard Brownian motion with respect to the observation filtration $\mathbb F^X$; see, for instance, \cite[Proposition 2.30 on p. 33]{BainCrisan}.

In the Bernoulli case, an explicit Bayes calculation (e.g., \cite[p. 9]{EksKarVai22}) gives
\begin{equation}\label{def_G}
    \widehat B(t)=G\big(X(t)-x\big),
    \qquad \text{ with } \qquad
    G(z)\coloneqq
    \frac{pe^z-(1-p)e^{-z}}{pe^z+(1-p)e^{-z}},
    \qquad z\in\mathbb R,
\end{equation}
or equivalently \(G(z)=\tanh(z-y_0)\), where
\begin{equation}\label{def_y0}
    y_0\coloneqq \frac12\log\left(\frac{1-p}{p}\right)
\end{equation}
is the unique zero of \(G\). We shall use repeatedly the elementary properties
\begin{equation}\label{G_properties}
    |G|<1,\qquad
    G'>0,\qquad
    G''+2GG'=0,\qquad
    G^2+G'=1.
\end{equation}
In particular, the function \(G\) of \eqref{def_G} is strictly convex on \((-\infty,y_0)\) and strictly concave on
\((y_0,\infty)\).

\begin{remark}
Although in the Bernoulli case one could equivalently parametrize the filter by the posterior probability \(\Pi(t)\coloneqq\mathbb P(B=1\mid\fc^X(t))\), since \(\widehat B(t)=2\Pi(t)-1\), we deliberately formulate the problem in terms of the posterior drift \(\widehat B\). This parametrization places the filtered quantity directly in the original dynamics and the infinitesimal generator; whereas, in combination with the explicit representation \eqref{def_G} and the identities \eqref{G_properties}, makes the subsequent variational and free-boundary analysis considerably more explicit than is typical in optimal stopping problems under incomplete information. It also seems better suited to extensions beyond two-point priors: if \(\bm{\pi}_t(db)\) denotes the conditional law of \(B\) given \(\fc^X(t)\), then the filtered drift is its posterior mean \(\widehat B(t)=\int b\,\bm{\pi}_t(db)\). For a general prior, this coefficient will typically depend on both time and displacement, and the special identities \eqref{G_properties} will no longer be available; nevertheless, we expect the filtering, centering, and Markovian-embedding steps of the present approach to remain robust.
\end{remark}

The representation \eqref{def_G} shows that the posterior drift depends on the observations only through its displacement from the initial position. We therefore introduce the process
\[
    Y(t)\coloneqq X(t)-x,\qquad 0\le t<\infty.
\]
Since \(X(\cdot)\) and \(Y(\cdot)\) generate the same filtration, the class of admissible stopping times is unchanged. In the new coordinate, \eqref{filtered_sde} becomes
\begin{equation}\label{sde_Y}
    Y(t)=\int_0^t G(Y(s))\,ds+N(t),
    \qquad 0\le t<\infty;
\end{equation}
and, with \(\T^Y=\T^X\) the collection of stopping times of the common observation filtration, the original value function \eqref{value_function_def} takes the form
\begin{equation}\label{value_func_via_Y}
    V(x)
    =
    \inf_{\tau\in\T^Y}
    \ex\Big[k\big(x+Y(\tau)\big)+c\,\tau\Big].
\end{equation}

\noindent \textit{Markovian Embedding:} It is useful to embed \eqref{value_func_via_Y} into a Markovian family, in which the displacement process is allowed to start from an arbitrary point. Accordingly, for \(y\in\mathbb R\), we let \(Y_y(\cdot)\) be the solution of
\begin{equation}\label{sde_Y_arbitrary}
    Y_y(t)=y+\int_0^t G(Y_y(s))\,ds+N(t),
    \qquad 0\le t<\infty,
\end{equation}
and note that, since \(y \in \mathbb{R}\), the filtrations generated by \(Y_y(\cdot)\) and \(N(\cdot)\) coincide.
Thus, the notation \(\T^Y\) for the corresponding stopping times is unambiguous, and the resulting value function is
\begin{equation}\label{2d_value_func_def}
    U(x,y)
    \coloneqq
    \inf_{\tau\in\T^Y}
    \ex\Big[k\big(x+Y_y(\tau)\big)+c\,\tau\Big],
    \qquad (x,y)\in\mathbb R^2.
\end{equation}
The original problem is recovered by restricting this two-parameter value function to the horizontal line \(y=0\), via 
\(
    V(x)=U(x,0),\, x\in\mathbb R.
\)
The advantage of the problem \eqref{2d_value_func_def} is that the dynamics of \(Y_y(\cdot)\) do not involve \(x\). Thus, for fixed \(x\), the embedded problem is a one-dimensional stopping problem in the displacement variable \(y\), with \(x\) acting as a parameter. We first solve this fixed-\(x\) problem in Section~\ref{sec:one_d_problem}; then study how these one-dimensional solutions fit together in the \((x,y)\)-plane in Section~\ref{sec:structure}; and finally return to the original problem of \eqref{value_function_def} in the \(x\)-variable in Section~\ref{sec:original_problem}.

\section{The Embedded One-Dimensional Problem}\label{sec:one_d_problem}

In this section, we solve the embedded stopping problem
\eqref{2d_value_func_def} with the parameter \(x\) fixed. This fixed-\(x\) problem already captures the main free-boundary aspects. Under suitable structural assumptions on the terminal cost  (for instance, under the sufficient condition \ref{ass_2}), we show that the continuation region is either empty or a single bounded interval. In the latter case, the endpoints are characterized by a free-boundary system, and the value function has a fairly-explicit form.

\noindent \textit{Preview:} In the first subsection we introduce the fixed-\(x\) notation, and the relevant state space and variational inequalities. In the second subsection we prove the verification result. A key point there is an \textit{a priori} boundedness statement, which shows from first principles that immediate stopping is optimal for all sufficiently large values of the displacement variable. 
The remaining analysis is split into three parts: in the third subsection we introduce and study the function \(\Phi_x\) of \eqref{def_phi}, which controls the geometry of the free boundaries. In the fourth subsection we construct the fixed-\(x\) solution under the resulting one-interval geometry. Finally, in the fifth subsection we present a convenient sufficient condition, formulated directly in terms of the terminal cost, which guarantees this geometry. The main result of the section is Theorem
\ref{th_fixed_x_solution}.

\subsection{Fixed-\(x\) formulation and variational inequalities}\label{subsec_var_ineqaulities}

Throughout this subsection, we fix \(x\ge0\). This entails no loss of generality for the original problem, because the case \(x<0\) is obtained from the symmetry of the terminal cost by changing \(p\) to \(1-p\). Starting from \(y\ge -x\), it is also enough to consider the process \(Y_y(\cdot)\) up to its first hitting time of the point \(-x\). Indeed, since \(k(x+\cdot)\) has its minimum at \(-x\), replacing any stopping time \(\sigma\) by \(\sigma\wedge\tau_{-x}\) cannot increase the expected cost, where \begin{equation}\label{def_hitting_time}
    \tau_{\xi}\coloneqq \inf\{t\ge0:Y_y(t)\le \xi\}.
\end{equation}
Thus, throughout this subsection, the natural state space will be
\(
I_x\coloneqq [-x,\infty).
\)
For convenience, we shall use the shorthand notation
\(U_x(\cdot)\coloneqq U(x,\cdot)\) and \(k_x(\cdot)\coloneqq k(x+\cdot)\).
We shall use the same convention for several auxiliary objects introduced below, and return to the full two-dimensional notation when the dependence on \(x\) becomes important. For instance, we shall denote here the \textit{stopping} and \textit{continuation} regions by
\begin{equation}\label{def_continuation_set}
    \mathcal S_x
    \coloneqq
    \{y\in I_x:U_x(y)=k_x(y)\}
    \quad \text{ and } \quad
    \mathcal C_x
    \coloneqq
    \{y\in I_x:U_x(y)<k_x(y)\}.
\end{equation}

We solve the problem \eqref{2d_value_func_def} using variational inequalities. To formulate those, and with \(G\) defined in \eqref{def_G}, we denote the infinitesimal generator of the diffusion in \eqref{sde_Y_arbitrary} by
\begin{equation}\label{def_generator_fixed_x}
    \mathcal L f(y)
    \coloneqq
    \frac12 f''(y)+G(y)f'(y), \qquad f\in C^2((-x,\infty)).
\end{equation} 
The obstacle problem associated with our stopping problem leads to the following variational inequalities. A sufficiently regular candidate function \(u(\cdot)\) should satisfy, wherever its derivatives $u'(\cdot)$ and \(u''(\cdot)\) exist, the system
\begin{enumerate}[label=\textbf{(\roman*)}]
    \item
    \(
        k_x(y)-u(y)\ge0;
    \)
    \label{var_ineq_1}

    \item
    \(
        \mathcal L u(y)+c
        =
        \frac12 u''(y)+G(y)u'(y)+c
        \ge0;
    \)
    \label{var_ineq_2}

    \item
    \(
        \big(k_x(y)-u(y)\big)\big(\mathcal L u(y)+c\big)=0.
    \)
    \label{var_ineq_3}
\end{enumerate}
These are the standard obstacle-type variational inequalities for optimal stopping problems; see \cite{ShiryaevRules}. 
The analysis can now be split into two essentially independent parts. In the first, we have to solve the above system of inequalities and provide a candidate for the value function. In the second, we show that, if a function satisfies \ref{var_ineq_1}--\ref{var_ineq_3}, it is indeed the value function of the problem. Since the verification part is shorter and more direct, we start with it, and only then solve the system.

\subsection{Verification}

We start with a useful a priori bound which shows that, for each fixed \(x\), the continuation region cannot extend arbitrarily far to the right. The proof uses only the scale function of the diffusion \(Y_y(\cdot)\) and does not rely on free-boundary considerations.

\begin{Prop}\label{prop_stopping_for_large_y}
    Fix \(x\ge0\). Immediate stopping is optimal for all sufficiently large displacement values \(y\); more precisely, there exists a threshold \(y_x^* \in (-x, \infty)\) such that
    \[
        U_x(y)=k_x(y),\qquad y\ge y_x^*.
    \]
\end{Prop}

\begin{proof}
    We claim that it suffices to establish the following auxiliary statement: there exists $y^*_x$ large enough such that, for all $y \ge y^*_x$ and all stopping times $\tau \in \mathcal{T}^Y$ with finite expectation and satisfying $\tau \le \tau_{-x}$ as in \eqref{def_hitting_time}, we have
    \begin{equation}\label{auxiliary_large_y_submartingale}
        \ex[Y_y(\tau)]\ge y.
    \end{equation}

    Indeed, assume that \eqref{auxiliary_large_y_submartingale} has been established for some $y^*_x$. Fix any stopping time \(\tau\in\T^Y\). Our goal is to show that immediate stopping is not worse than choosing \(\tau\). We might assume that $\tau$ has finite expectation, and that $\tau \le \tau_{-x}$ holds almost surely in the notation of \eqref{def_hitting_time}. Indeed, if $\ex [\tau] = \infty$, the expected cost is infinite and there is nothing to prove; whereas replacing \(\tau\) by \(\tau\wedge\tau_{-x}\) cannot increase the cost, as \(k_x(\cdot)\) is minimized at \(-x\).
    Now, for any such \(\tau\) the inequality \eqref{auxiliary_large_y_submartingale}, the convexity of \(k_x(\cdot)\), and the fact that \(k_x(\cdot)\) is non-decreasing on \(I_x\), give
    \[
        \ex\big[k_x(Y_y(\tau))+c\tau\big]
        \ge
        \ex\big[k_x(Y_y(\tau))\big]
        \ge
        k_x\big(\ex[Y_y(\tau)]\big)
        \ge
        k_x(y).
    \]
    Immediate stopping attains the value \(k_x(y)\), so \(U_x(y)=k_x(y)\) holds for all \(y\ge y_x^*\), which is exactly the desired statement. Thus, it suffices to prove \eqref{auxiliary_large_y_submartingale}.

    To do so, let \(S\) be the scale function of \(Y_y(\cdot)\), given by
    \begin{equation}\label{scale_function}
        S(y)
        \coloneqq
        \int_{y_0}^{y}
        \exp\left(
            -2\int_{y_0}^{\xi}G(\zeta)\,d\zeta
        \right)d\xi,
    \end{equation}
    where $y_0$ is defined in \eqref{def_y0}.
    Note that
    \(
        S'(y)>0,
        S''(y)=-2G(y)S'(y),
    \)
    and, with the present normalization, \(S\equiv G\). In particular, \(S\) is
    bounded, strictly increasing, strictly convex on \((-\infty,y_0)\), and
    strictly concave on \((y_0,\infty)\).

    We claim that there exists \(y_x^*>-x\) such that, for every \(y\ge y_x^*\),
    the tangent line
    \(
        L_y(z)\coloneqq S(y)+S'(y)(z-y)
    \)
    dominates \(S\) on the whole interval \(I_x=[-x,\infty)\):
    \begin{equation}\label{scale_supporting_line}
        S(z)\le L_y(z),\qquad z\ge -x.
    \end{equation}
    To see this, choose \(y_x^*\ge y_0\vee(-x)\) so large that
    \[
        L_y(-x)\ge S(-x),\qquad y\ge y_x^*.
    \]
    This is possible because \(S(y)\to1\) and \(S'(y)(y+x)\to0\) as \(y\to\infty\), whereas \(S(-x)<1\). For such large \(y\), concavity of \(S\) on \([y_0,\infty)\) implies that \(L_y\) lies above \(S\) on \([y_0,\infty)\). If \(-x<y_0\), then \(S\) is convex on \([-x,y_0]\); and \(L_y\) dominates \(S\) at both endpoints \(-x\) and \(y_0\), thus on the whole interval \([-x,y_0]\) as well. The geometry is illustrated in Figure~\ref{fig:scale_function}.

    \begin{figure}[ht]
        \centering
        \includegraphics[width=.75\linewidth]{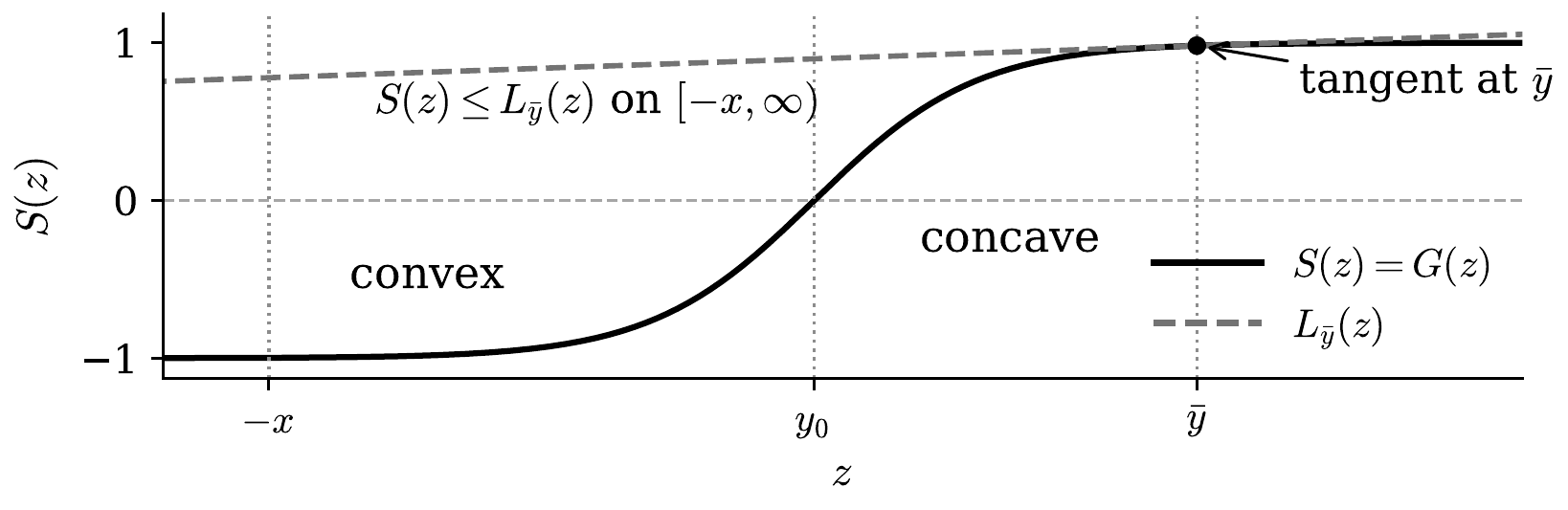}
        \caption{The scale function and the supporting tangent line.}
        \label{fig:scale_function}
    \end{figure}

    This proves
    \eqref{scale_supporting_line}. Now we fix \(y\ge y_x^*\), and let \(\tau\le\tau_{-x}\) be a stopping time with finite expectation. Since \(S(Y_y(\cdot))\) is a bounded martingale (see, e.g., \cite[Propositions 3.4-3.5 on pp. 302-303]{RevuzYor}), optional sampling gives
    \(
        \ex\big[S(Y_y(\tau))\big]=S(y).
    \)
    Moreover, \(Y_y(\tau)\) is integrable, because \(|G|\le1\) and \(\ex[\tau]<\infty\). 
    Thus, applying \eqref{scale_supporting_line} to \(Y_y(\tau)\) and taking expectations, we obtain
    \[
        S(y)
        =
        \ex\big[S(Y_y(\tau))\big]
        \le
        S(y)+S'(y)\big(\ex[Y_y(\tau)]-y\big).
    \]
    Since \(S'(y)>0\), this yields \eqref{auxiliary_large_y_submartingale} and completes the proof.
\end{proof}

\begin{Cor}\label{cor_bounded_st_times}
    Fix \(x\ge0\), and let \(y_x^*\) be as in Proposition \ref{prop_stopping_for_large_y}. In the fixed-\(x\) problem \eqref{2d_value_func_def} with value $U_x(\cdot)$, it is enough to minimize over stopping times \(\tau\) satisfying
    \begin{equation}\label{def_double_hitting_time}
        \tau\le \tau_{-x,y_x^*} \quad 
        \text{ a.s., where }
        \qquad
        \tau_{\xi,\zeta}
        \coloneqq
        \inf\{t\ge0:Y_y(t)\notin(\xi,\zeta)\}.
    \end{equation}
\end{Cor}

\begin{proof}
    As above, replacing a stopping time by its minimum with \(\tau_{-x}\) cannot increase the cost. If the process reaches \(y_x^*\), Proposition \ref{prop_stopping_for_large_y} and the strong Markov property imply that immediate stopping is optimal from that point onward. Thus, continuing beyond the first exit from \((-x,y_x^*)\) cannot improve the expected cost.
\end{proof}

We state now the verification result in the form needed below.

\begin{Prop}\label{prop_verification}
    Fix \(x\ge0\). 
    Let \(\widehat U:I_x\to\mathbb R\) be continuous on \(I_x\) and \(C^1\) on \((-x,\infty)\). Assume that \(\widehat U'\) extends to \(I_x\) and is absolutely continuous on every compact interval \([-x,r]\), \(r>-x\). Suppose that \(\widehat U\) satisfies \ref{var_ineq_1}--\ref{var_ineq_3} of subsection \ref{subsec_var_ineqaulities} on \(I_x\), wherever \(\widehat U''\) exists. Suppose also that the candidate continuation region
    \[
        \widehat{\mathcal C}
        \coloneqq
        \{y\in I_x:\widehat U(y)<k_x(y)\}
    \]
    is contained in \((-x,R)\) for some \(R<\infty\). Then
    \(
        \widehat U(y)=U_x(y),\ y\in I_x,
    \)
    and an optimal stopping time is given by 
    \[
        \tau^*
        \coloneqq
        \inf\{t\ge0:Y_y(t)\notin\widehat{\mathcal C}\}.
    \]
\end{Prop}

\begin{proof}
    We first show that \(\widehat U\) is a lower bound on the achievable cost.
    By Corollary \ref{cor_bounded_st_times}, it is enough to consider stopping times \(\tau\) satisfying \(\tau\le\tau_{-x, \, y_x^*}\) a.s. Such stopping times have finite expectation, since the exit time of \(Y_y(\cdot)\) from a bounded interval has finite expectation.
    For any such stopping time $\tau$, applying Itô's formula to \(\widehat U(Y_y(\tau))\) in the generalized form valid for functions with absolutely continuous first derivative (see, e.g., \cite[Problem 7.3 on p. 219]{BMSC}), gives
    \begin{equation}\label{Ito}
        \begin{split}
            \widehat U(Y_y(\tau))-\widehat U(y)
            &=
            \int_0^\tau \widehat U'(Y_y(s))\,dN(s)
            +
            \int_0^\tau
            \left(
                \frac12\widehat U''(Y_y(s))
                +G(Y_y(s))\widehat U'(Y_y(s))
            \right)ds
            \\
            &\ge
            \int_0^\tau \widehat U'(Y_y(s))\,dN(s)-c\tau,
        \end{split}
    \end{equation}
    where the inequality follows from \ref{var_ineq_2}. Since
    \(\tau\le\tau_{-x,y_x^*}\), the process remains, up to time \(\tau\), in the compact interval
    \([-x,y_x^*]\), where \(\widehat U'\) is bounded. Therefore, the local martingale
    \begin{equation}
        M(t) \coloneqq \int\limits_{0}^{\tau \wedge t} \widehat{U}'(Y_y(s)) \, dN(s), \quad 0 \le t < \infty,
    \end{equation}
    has quadratic variation 
    \begin{equation}\label{stoch_integral_quadratic_variation}
        \langle M \rangle (t)=
            \int\limits_{0}^{\tau \wedge t} \Big(\widehat{U}'(Y_y(s))\, \Big)^2 \, ds,
    \end{equation}
    which satisfies $\mathbb{E}[\langle M \rangle(\infty)] < \infty$ thanks to $\ex[\tau] < \infty$. Thus, there exists a real constant $C > 0$ with
    \begin{equation}
        \ex [M^2(t)] \le \ex [\langle M \rangle(t)] \le \mathbb{E}[\langle M \rangle(\infty)] \le C \cdot \ex [\tau] < \infty, \quad \forall \, t \ge 0,
    \end{equation}
    where the first inequality is well-known for local martingales (see, e.g., p. 38 in \cite{BMSC}), and the second follows from the monotonicity of the quadratic variation. Hence, $M(\cdot)$ is a bounded in $\mathbb{L}^2$ martingale.
    
    Taking expectations on both sides of \eqref{Ito} and applying optional sampling for martingales bounded in $\mathbb{L}^2$,
    gives
    \(
        \widehat U(y)
        \le
        \ex\big[\widehat U(Y_y(\tau))+c\tau\big]
        \le
        \ex\big[k_x(Y_y(\tau))+c\tau\big];
    \)
    and, taking the infimum over all admissible stopping times, we get
    \begin{equation}\label{value_function_as_lower_bound}
        \widehat U(y)\le U_x(y),\qquad y\in I_x.
    \end{equation}

    It remains to prove the reverse inequality. If \(y\notin\widehat{\mathcal C}\), then \(\widehat U(y)=k_x(y)\) and \(\tau^*=0\), so there is nothing to prove. Suppose therefore that \(y\in\widehat{\mathcal C}\). 
    Since \(\widehat{\mathcal C}\subset(-x,R)\), the time \(\tau^*\) is bounded above by the first exit time of \(Y_y(\cdot)\) from the finite interval \((-x,R)\), and hence has finite expectation.
    On \(\widehat{\mathcal C}\), the strict inequality \(\widehat U<k_x\) and the complementarity condition \ref{var_ineq_3} imply
    \(
        \mathcal L\widehat U+c=0.
    \)
    Therefore, applying Itô's formula up to \(\tau^*\) and recalling that the stochastic integral has expectation zero, we obtain from \eqref{Ito} the equality
    \(
        \widehat U(y)
        =
        \ex\big[\widehat U(Y_y(\tau^*))+c\tau^*\big].
    \)
    By continuity, \(Y_y(\tau^*)\) belongs to the contact set, therefore
    \(\widehat U(Y_y(\tau^*))=k_x(Y_y(\tau^*))\) and
    \[
        \widehat U(y)
        =
        \ex\big[k_x(Y_y(\tau^*))+c\tau^*\big]
        \ge
        U_x(y).
    \]
    Together with \eqref{value_function_as_lower_bound}, this proves \(\widehat U=U_x\), and the same identity shows that \(\tau^*\) is optimal.
\end{proof}

\subsection{The crucial function \texorpdfstring{\(\Phi_x\)}{Phi\_x}, and the one-interval geometry}

We turn to the variational inequalities \ref{var_ineq_1}--\ref{var_ineq_3}. On any open interval where continuation is expected, the complementarity condition \ref{var_ineq_3} forces the candidate value function to satisfy
\begin{equation}\label{continuation_ode_fixed_x}
    \frac12 u''(y)+G(y)u'(y)+c=0.
\end{equation}
Using the identities for \(G\) in \eqref{G_properties}, one checks directly
that the general solution of \eqref{continuation_ode_fixed_x} is
\begin{equation}\label{U_general_form}
    u(y)=A\,G(y)+B-cyG(y)
\end{equation}
for appropriate real constants \(A,B\). This is because the functions $G, 1$ are linearly independent solutions of the homogeneous version of the equation in \eqref{continuation_ode_fixed_x}, and the function $-cyG(y)$ is a particular solution of \eqref{continuation_ode_fixed_x}; recall the identities \eqref{G_properties}.
Thus, the fixed-\(x\) problem reduces to the question of how to glue functions of the form \eqref{U_general_form} to the obstacle \(k_x(\cdot)\).

To determine the constants $A$ and $B$ of \eqref{U_general_form} from such gluing, we use value-matching and the smooth-fit principle. For instance, suppose that \(u(\cdot)\) has the form \eqref{U_general_form} and is ``smoothly pasted'' to \(k_x(\cdot)\) at an interior point \(y>-x\), meaning \(u'(y)=k_x'(y)\); then we have
\(
    A
    =
    \big(k_x'(y)+cG(y)\big)/G'(y)+cy.
\)
This motivates the definitions
\begin{equation}\label{def_phi}
    \Phi_x(y)
    \coloneqq
    \frac{k_x'(y)+cG(y)}{G'(y)}+cy,
    \ y>-x,
    \qquad 
    \Phi_x(-x)
    \coloneqq
    \frac{k'(0+)+cG(-x)}{G'(-x)}-cx .
\end{equation}
In particular, if a continuation interval has two interior free endpoints \(a<b\), then smooth fit at both endpoints forces \(\Phi_x(a)=\Phi_x(b)\), since both need to be equal to $A$. If the lower endpoint is the boundary point \(-x\), then no lower smooth-fit condition is imposed there.

The derivative of the function \(\Phi_x\) in \eqref{def_phi} has a particularly useful form. Direct calculation, using \(G''+2GG'=0\) and \(G^2+G'=1\) from \eqref{G_properties} and the notation of \eqref{def_generator_fixed_x}, gives
\begin{equation}\label{phi_derivative}
    \Phi_x'(y)
    =
    \frac{2}{G'(y)}
    \left(
        \frac12 k_x''(y)+G(y)k_x'(y)+c
    \right)
    =
    \frac{2}{G'(y)}
    \big(
        \mathcal L k_x (y)+c
    \big),
    \qquad y>-x.
\end{equation}
Since \(G'>0\), the sign of \(\Phi_x'\) is the same as the sign of \(\mathcal L k_x+c\). Thus, the function \(\Phi_x\) records precisely where immediate stopping satisfies, or fails to satisfy, the second variational inequality.

We first dispose of the case where \(\Phi_x\) has no interval of decrease.

\begin{Prop}\label{prop_Phi_prime}
    If \(\,\Phi_x\) is non-decreasing on \((-x,\infty)\), then \(U_x=k_x\) on \(I_x\); equivalently, \(\mathcal C_x=\varnothing\).
\end{Prop}

\begin{proof}
    By \eqref{phi_derivative}, the non-decrease of \(\Phi_x\) implies
    \(
        \mathcal L k_x(y)+c
        = k_x''(y) / 2+G(y)k_x'(y)+c
        \ge0
    \)
    for all $y>-x$,
    thus \(\widehat U\equiv k_x\) satisfies
    \ref{var_ineq_1}--\ref{var_ineq_3}.
    The regularity required by Proposition \ref{prop_verification} follows from assumption \ref{ass_1}, with
    \(
    k_x'(-x)\coloneqq k'(0+).
    \) 
    Its candidate continuation set is empty, hence bounded. Proposition \ref{prop_verification} gives \(\widehat U=U_x\), and so immediate stopping is then optimal everywhere.
\end{proof}

Conversely, points at which \(\Phi_x\) is strictly decreasing necessarily belong to the continuation region. This observation is not needed for verification itself, but explains why the geometry of the function \(\Phi_x\) in \eqref{def_phi} controls the geometry of the stopping problem.

\begin{Lem}\label{lem_decrease_points_in_continuation}
    If \(y>-x\) and \(\Phi_x'(y)<0\), then \(y\in\mathcal C_x\).
\end{Lem}

\begin{proof}
    By \eqref{phi_derivative}, the inequality \(\Phi_x'(y)<0\) is equivalent to
    \(
        (1/2) k_x''(y)+G(y)k_x'(y)+c<0.
    \)
    By continuity, there exist \(\delta>0\) and \(\varepsilon>0\) such that
    \(y-\delta>-x\) and
    \[
        \frac12 k_x''(z)+G(z)k_x'(z)+c\le -\varepsilon,
        \qquad z\in(y-\delta,y+\delta).
    \]
    For \(h>0\), we set
    \(
        \sigma_h
        \coloneqq
        h\wedge \inf\{t\ge0:Y_y(t)\notin(y-\delta,y+\delta)\},
    \)
    and observe by Dynkin's formula
    \[
        \ex\big[k_x(Y_y(\sigma_h))+c\sigma_h\big]-k_x(y)
        =
        \ex\int_0^{\sigma_h}
        \left(
            \frac12 k_x''(Y_y(t))
            +G(Y_y(t))k_x'(Y_y(t))
            +c
        \right)dt
        \le
        -\varepsilon\,\ex[\sigma_h]<0.
    \]
    Hence, waiting until the exit time \(\sigma_h\) from the interval $(y-\delta, y+\delta)$ performs strictly better than stopping
    immediately, and so \(U_x(y)<k_x(y)\).
\end{proof}

The preceding proposition and lemma show that the only non-trivial case is when \(\Phi_x\) of \eqref{def_phi} decreases somewhere. Since \(\Phi_x'\) is continuous on \((-x,\infty)\), the set
\(
    N_x\coloneqq \{y\in(-x,\infty):\Phi_x'(y)<0\}
\)
is open. In general, this open set need not be connected: it may be a union of several disjoint intervals, and the corresponding continuation region may have a more complicated multi-interval structure. Lemma \ref{lem_decrease_points_in_continuation} only shows that \(N_x\subseteq \mathcal C_x\); it does not rule out additional continuation components, nor does it identify the stopping region.

\noindent \textit{Single Valley:} The remainder of this section focuses on the one-interval regime. We assume
\begin{equation}\label{one_interval_phi_geometry}
    N_x
    \coloneqq
    \{y\in(-x,\infty):\Phi_x'(y)<0\}
    =
    (\ell_x,r_x)
\end{equation}
for some \(-x\le \ell_x<r_x<\infty\), and we allow \(\ell_x=-x\), in which case the interval is understood as \((-x,r_x)\). The finiteness of \(r_x\) follows from Proposition \ref{prop_stopping_for_large_y} and Lemma \ref{lem_decrease_points_in_continuation}.

The assumption \eqref{one_interval_phi_geometry} is a substantive structural restriction. It is, however, the regime in which the free-boundary problem has a tractable canonical solution; moreover, the sufficient condition in subsection \ref{subsec_single_valley} guarantees this geometry for the main examples considered below, such as terminal costs of power and exponential types.
Under the assumption \eqref{one_interval_phi_geometry}, the graph of the function \(\Phi_x\) in \eqref{def_phi} has a \textit{single valley}: it is non-decreasing on \((-x,\ell_x]\), strictly decreasing on \((\ell_x,r_x)\), and non-decreasing again on \([r_x,\infty)\); see Figure \ref{fig:phi_one_valley} below. We show now that this single-valley geometry leads to a single continuation interval \((a_x,b_x)\), and identify both endpoints.

\subsection{Construction of the fixed-\texorpdfstring{\(x\)}{x} solution}\label{subsec_main_construction}

We construct now the value function under the single-valley geometry \eqref{one_interval_phi_geometry}, guided by the smooth-fit discussion above. If \((a,b)\) is a continuation interval whose lower endpoint is interior to \([-x, \infty)\), then smooth fit at \(a\) and \(b\) requires the coefficient \(A\) in \eqref{U_general_form} to satisfy \(A=\Phi_x(a)=\Phi_x(b)\). If the lower endpoint is the state-space boundary \(-x\), then no smooth fit is imposed there, but smooth fit at the upper endpoint still gives \(A=\Phi_x(b)\). Thus, it is natural to parametrize candidate intervals by the ``smooth-fit'' level \(q\), as explained below.

Recalling the notation of \eqref{def_phi}, \eqref{one_interval_phi_geometry}, we set now
\begin{equation}\label{def_q_bounds_fixed_x}
    \underline q_x\coloneqq \Phi_x(r_x)
    <
    \Phi_x(\ell_x) \eqqcolon \overline q_x.
\end{equation}
For each \(q\in(\underline q_x,\overline q_x)\), we define two points \(a_x(q)\) and \(b_x(q)\) as follows. 

The upper endpoint is the first point on the right of \(r_x\) at which \(\Phi_x\) reaches the level \(q\):
\begin{equation}\label{def_bq_fixed_x}
    b_x(q)
    \coloneqq
    \inf\{y\ge r_x:\Phi_x(y)\ge q\}.
\end{equation}
Note that $b_x(q)$ is well-defined, since \(\Phi_x(y)\to+\infty\) as \(y\to\infty\). This follows from the definition of $\Phi_x$ in \eqref{def_phi} and the facts that \(G'(y)\to0\), \(G(y)\to1\) as \(y\to\infty\), and \(k_x'(y)\ge0\) for all sufficiently large \(y\). Hence, every level \(q\in(\underline q_x,\overline q_x)\) is reached on the right branch of
\(\Phi_x\).

The lower endpoint is defined by
\begin{equation}\label{def_aq_fixed_x}
    a_x(q)
    \coloneqq
    \begin{cases}
        \sup\{y\in[-x,\ell_x]:\Phi_x(y)\le q\},
        & q\ge\Phi_x(-x),\\[0.4em]
        -x,
        & q< \Phi_x(-x).
    \end{cases}
\end{equation}
Thus, \(a_x(q)\le \ell_x<r_x\le b_x(q)\), and \(\Phi_x(b_x(q))=q\). If \(a_x(q)>-x\), then also \(\Phi_x(a_x(q))=q\), so the same smooth-fit level \(q\) is imposed at both endpoints. If \(a_x(q)=-x\), then the lower endpoint is the boundary of the state space and no lower smooth-fit condition is imposed.

For a fixed level \(q\), suppose that the continuation value on
\((a_x(q),b_x(q))\) is given by
\(
    u_q(y)=qG(y)+R-cyG(y),
\)
as in \eqref{U_general_form}.
The constant \(R\) is determined by value-matching at the lower endpoint:
\[
    R
    =
    k_x(y)-qG(y)+c\,yG(y) \Big\vert_{y = a_x(q)}.
\]
With this choice, \(u_q(a_x(q))=k_x(a_x(q))\). 

To impose value-matching at the upper endpoint, we define the gap
\[
    F_q(y)\coloneqq k_x(y)-u_q(y),
\]
and note \(F_q(a_x(q))=0\), as well as
\(
    F_q'(y)
    =
    \big(\Phi_x(y)-q\big)G'(y),
\)
using the definition of \(\Phi_x\) in \eqref{def_phi}.
Therefore,
\[
    F_q(b_x(q))
    =
    \int_{a_x(q)}^{b_x(q)}\big(\Phi_x(z)-q\big)\,dG(z),
\]
so value matching at the upper endpoint is equivalent to the vanishing of
the balancing functional
\begin{equation}\label{def_Jx}
    J_x(q)
    \coloneqq
    \int_{a_x(q)}^{b_x(q)}\big(\Phi_x(z)-q\big)\,dG(z),
    \qquad
    q\in(\underline q_x,\overline q_x).
\end{equation}

The construction is illustrated in Figure~\ref{fig:phi_one_valley}. The level \(q\) determines the two candidate endpoints \(a_x(q)\) and \(b_x(q)\), and the balancing equation \(J_x(q)=0\) selects the unique level \(q_x\) for which the positive and negative signed \(dG\)-areas cancel.

\begin{figure}[ht]
    \centering
    \includegraphics[width=.72\linewidth]{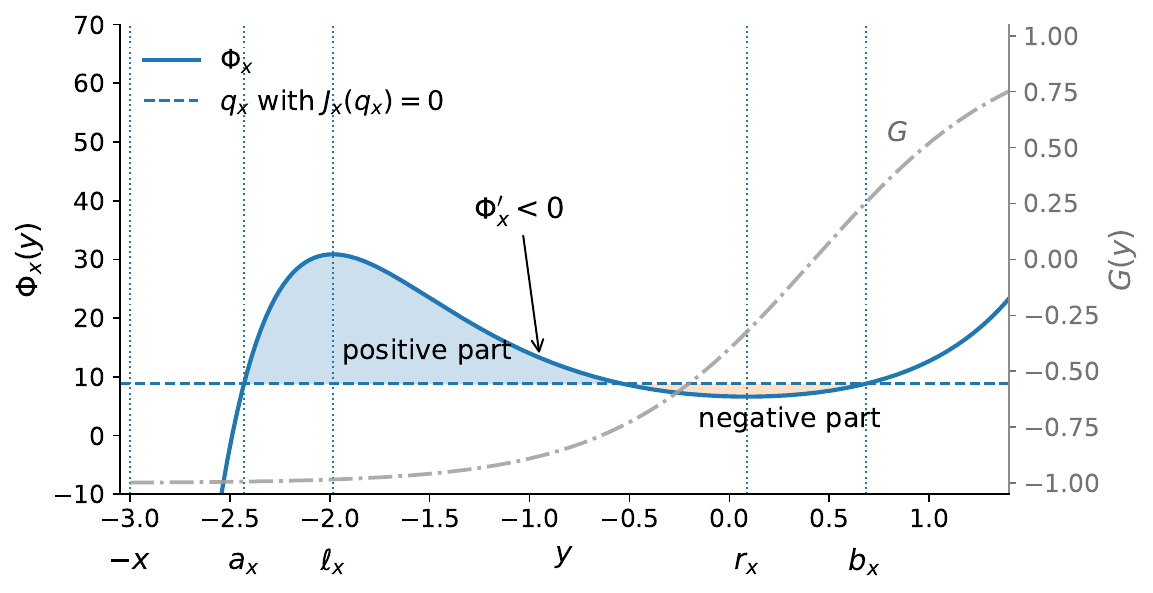}
    \caption{The one-valley geometry of \(\Phi_x\) and the balancing level
    \(q_x\) for $k(s) = s^2, \, c = 1, \, p = 0.3, \, x = 3$. The shaded regions represent the positive and negative parts of the signed \(dG\)-area in \(J_x(q_x)\).}
    \label{fig:phi_one_valley}
\end{figure}

\begin{Lem}\label{lem_uniqueness}
    Under the one-interval geometry \eqref{one_interval_phi_geometry}, and recalling \eqref{def_Jx}, there
    exists a unique \(q_x\in(\underline q_x,\overline q_x)\) such that
    \begin{equation}\label{eq_main_balancing_condition}
        J_x(q_x)=0.
    \end{equation}
\end{Lem}

\begin{proof}
    The endpoint maps \(a_x(\cdot)\) and \(b_x(\cdot)\) are non-decreasing, and
    \(J_x\) is continuous on \((\underline q_x,\overline q_x)\). This follows
    directly from the one-valley geometry; possible flat portions of
    \(\Phi_x\) do not affect the value of the integral, since the integrand
    \(\Phi_x-q\) vanishes on them.

    As \(q\downarrow\underline q_x\), the negative part of the signed area in \(J_x(q)\) shrinks to zero, while a strictly positive part remains across the valley. Hence
    \[
        \lim_{q\downarrow\underline q_x} J_x(q)>0.
    \]
    Similarly, as \(q\uparrow\overline q_x\), the positive part shrinks to zero, while a strictly negative part remains across the valley. Hence
    \[
        \lim_{q\uparrow\overline q_x} J_x(q)<0.
    \]
    By continuity, \(J_x\) has at least one zero.

    It remains to prove uniqueness, which will follow from the strict monotonicity of $J_x$ in \eqref{def_Jx}. Let \(\underline q_x<q_1<q_2<\overline q_x\), and write
    \[
        a_i\coloneqq a_x(q_i),
        \qquad
        b_i\coloneqq b_x(q_i),
        \qquad i=1,2.
    \]
    By construction, we have
    \(
        a_1\le a_2\le \ell_x<r_x\le b_1\le b_2.
    \)
    Decomposing the difference \(J_x(q_2)-J_x(q_1)\), we write
    \[
        \begin{split}
        J_x(q_2)-J_x(q_1)
        =
        &-\int_{a_1}^{a_2}
            \big(\Phi_x(z)-q_1\big)\,dG(z)
        \\&
        -(q_2-q_1)\big(G(b_1)-G(a_2)\big) 
        +\int_{b_1}^{b_2}
            \big(\Phi_x(z)-q_2\big)\,dG(z).
        \end{split}
    \]
    On \([a_1,a_2]\), we have \(\Phi_x\ge q_1\), so the first term is non-positive. On \([b_1,b_2]\), we have \(\Phi_x\le q_2\), so the third term is non-positive. The middle term is negative, because \(q_2>q_1\), \(G\) is strictly increasing, and \(a_2<b_1\). Therefore,
    \(
        J_x(q_2)<J_x(q_1).
    \)
    Thus, \(J_x\) is strictly decreasing on
    \((\underline q_x,\overline q_x)\), and its zero is unique.
\end{proof}

With \(q_x\) the unique solution of \eqref{eq_main_balancing_condition}, we
set
\begin{equation}\label{def_ax_bx_fixed_x}
    a_x\coloneqq a_x(q_x),
    \qquad
    b_x\coloneqq b_x(q_x),
\end{equation}
and show that these two points give the desired gluing of the continuation
solution to the obstacle. Define
\begin{equation}\label{def_Rx_fixed_x}
    R_x
    \coloneqq
    k_x(a_x)-q_xG(a_x)+c\,a_xG(a_x),
\end{equation}
and
\begin{equation}\label{def_wx_fixed_x}
    w_x(y)
    \coloneqq
    q_xG(y)+R_x-cyG(y),
    \qquad a_x<y<b_x.
\end{equation}

\begin{Lem}\label{lem_gluing_fixed_x}
    The function \(w_x\) satisfies
    \[
        w_x(a_x)=k_x(a_x),
        \qquad
        w_x(b_x)=k_x(b_x),
        \qquad 
        w_x'(b_x)=k_x'(b_x);
    \]
    and if \(a_x>-x\), then also
    \(
        w_x'(a_x)=k_x'(a_x).
    \)
    Finally,
    \begin{equation}\label{obstacle_ineq}
        w_x(y) < k_x(y),\qquad a_x < y < b_x.
    \end{equation}
\end{Lem}

\begin{proof}
    The value matching at \(a_x\) follows immediately from the definition of
    \(R_x\). The function
    \(
        F_x(y)\coloneqq k_x(y)-w_x(y),\, a_x\le y\le b_x
    \)
    satisfies \(F_x(a_x)=0\), and differentiation gives
    \begin{equation}\label{gap_derivative_fixed_x}
        F_x'(y)
        =
        \big(\Phi_x(y)-q_x\big)G'(y),
        \qquad a_x<y<b_x,
    \end{equation}
    thus also
    \begin{equation}\label{F_integral}
        F_x(y)
        =
        \int_{a_x}^{y}\big(\Phi_x(z)-q_x\big)\,dG(z),
        \qquad a_x\le y\le b_x.
    \end{equation}
    In particular, the balancing condition \(J_x(q_x)=0\) gives \(F_x(b_x)=0\), which is value matching at \(b_x\).
    Since \(\Phi_x(b_x)=q_x\), \eqref{gap_derivative_fixed_x} gives \(F_x'(b_x)=0\), and hence \(w_x'(b_x)=k_x'(b_x)\). If \(a_x>-x\), then \(\Phi_x(a_x)=q_x\), and the same argument gives \(w_x'(a_x)=k_x'(a_x)\).

    It remains to prove the obstacle inequality \eqref{obstacle_ineq}. By the single-valley geometry, there exists a unique point \(\eta_x\in(\ell_x,r_x)\) such that \(\Phi_x(\eta_x)=q_x\). Moreover, the definitions of \(a_x\) and \(b_x\), together with the strict decrease of \(\Phi_x\) on \((\ell_x,r_x)\), give
    \[
        \Phi_x(z)>q_x
        \quad\text{for } \ \ a_x<z<\eta_x,
        \qquad
        \Phi_x(z)<q_x
        \quad\text{for } \ \ \eta_x<z<b_x.
    \]
    Since \(dG(z)=G'(z)\,dz\) is a positive measure, the representation of \(F_x\) in \eqref{F_integral} implies \(F_x(y) > 0\) for \(a_x < y\le\eta_x\). If \(\eta_x\le y < b_x\), then, using \(F_x(b_x)=0\), we obtain
    \[
        F_x(y)
        =
        -\int_y^{b_x}\big(\Phi_x(z)-q_x\big)\,dG(z)
         > 0,
    \]
    whence \(w_x(y) < k_x(y)\) on \((a_x,b_x)\).
\end{proof}

We can state now the fixed-\(x\) solution result.

\begin{Th}\label{th_fixed_x_solution}
    Fix \(x\ge0\). If the function \(\Phi_x\) of \eqref{def_phi} is non-decreasing on \((-x,\infty)\), then  immediate stopping is optimal for every \(y\in I_x\), and
    \(
        U_x(y)=k_x(y),\, y\in I_x.
    \)

    Suppose, on the other hand, that the one-interval regime \eqref{one_interval_phi_geometry} holds. Let \(q_x\) be the unique solution of \eqref{eq_main_balancing_condition}, and \(a_x,b_x,R_x\) be defined by \eqref{def_ax_bx_fixed_x}, \eqref{def_Rx_fixed_x}. Then
    \(
        \mathcal C_x=(a_x,b_x)
    \)
    is the continuation region; the value function is given by
    \begin{equation}\label{fixed_x_value_formula}
        U_x(y)
        =
        \begin{cases}
            k_x(y), & -x\le y\le a_x,\\[0.3em]
            q_xG(y)+R_x-cyG(y), & a_x<y<b_x,\\[0.3em]
            k_x(y), & y\ge b_x;
        \end{cases}
    \end{equation}
    and the stopping time
    \[
        \tau_x^*
        \coloneqq
        \inf\{t\ge0:Y_y(t)\notin(a_x,b_x)\},
    \]
    with the convention \(\tau_x^*=0\) when \(y\notin(a_x,b_x)\), is optimal.
\end{Th}

\begin{proof}
    The case where \(\Phi_x\) is non-decreasing is dealt with in Proposition \ref{prop_Phi_prime}. 
    Assume now that \eqref{one_interval_phi_geometry} holds, and define the candidate value function $\widehat U_x$ by the expression on the right-hand side of \eqref{fixed_x_value_formula}.
    By Lemma \ref{lem_gluing_fixed_x}, assumption \ref{ass_1}, and the piecewise definition \eqref{fixed_x_value_formula}, this function is continuous on \(I_x\), and its derivative extends to \(I_x\) and is absolutely continuous on every compact interval \([-x,r]\), \(r>-x\). It also satisfies the obstacle inequality
    \(
        \widehat U_x(y)\le k_x(y),\, y\in I_x.
    \)
    In fact, Lemma \ref{lem_gluing_fixed_x} shows that the inequality is strict exactly on \((a_x,b_x)\), where the function \(\widehat U_x\) has the form \eqref{U_general_form}, and therefore solves
    \(
        \mathcal L\widehat U_x(y)+c=0.
    \)
    Outside \((a_x,b_x)\), we have \(\widehat U_x=k_x\). Since \(a_x\le\ell_x<r_x\le b_x\), the complement of \((a_x,b_x)\) is contained in the region where \(\Phi_x\) is non-decreasing. Hence, by
    \eqref{phi_derivative},
    \[
        \mathcal L k_x(y)+c
        =
        \frac12 k_x''(y)+G(y)k_x'(y)+c
        =
        \frac{G'(y)}{2}\Phi_x'(y)
        \ge0
    \]
    wherever \(k_x''\) exists outside \((a_x,b_x)\). Thus \(\widehat U_x\) satisfies the variational inequalities \ref{var_ineq_1}--\ref{var_ineq_3}. Its candidate continuation set is the bounded interval \((a_x,b_x)\). Proposition \ref{prop_verification} therefore gives \(\widehat U_x=U_x\), and the first exit time from
    \((a_x,b_x)\) is optimal.
\end{proof}

\begin{remark}
    For each fixed \(x\), the stopping problem of this section falls within the classical potential-theoretic framework for optimal stopping of Markov processes developed by Dynkin; see \cite{Dynkin1963,Dynkin1965,DynkinYushkevich1969}. In the one-dimensional diffusion setting, it can also be placed within the scale-function and generalized-concavity framework developed by Dayanik and Karatzas \cite{DayanikKaratzas2003}, after the standard adjustment for the running cost. We use instead the direct variational-inequality and balancing construction above, since it keeps explicit the role of \(\Phi_x\) and the dependence of the continuation interval on \(x\), and is therefore better suited to the ``gluing'' analysis of Section~\ref{sec:structure}.
\end{remark}

\subsection{A sufficient condition for one-interval geometry}\label{subsec_single_valley}

Theorem \ref{th_fixed_x_solution} is formulated in terms of the geometry of the function \(\Phi_x\) in \eqref{def_phi}. We now present a convenient condition, expressed directly in terms of the terminal cost, which guarantees this geometry. This is the reason for the assumption \ref{ass_2} in the model section.

Fix \(x\ge0\), and write
\[
    I_x^\circ\coloneqq (-x,\infty),\qquad
    I_x^-\coloneqq I_x^\circ\cap(-\infty,y_0),
    \qquad
    I_x^+\coloneqq I_x^\circ\cap[y_0,\infty).
\]
The interval \(I_x^-\) may be empty. On \(I_x^\circ\), we define the functions
\begin{equation}\label{def_Qx_Thetax}
    P_x(y)
    \coloneqq
    \frac{k_x''(y)+2c}{2k_x'(y)},
    \qquad
    \Theta_x(y)
    \coloneqq
    P_x(y)+G(y).
\end{equation}
Whenever \(k_x'(y)>0\), the derivative identity \eqref{phi_derivative} can be rewritten as
\begin{equation}\label{phi_derivative_theta}
    \Phi_x'(y)
    =
    \frac{2k_x'(y)}{G'(y)}\,\Theta_x(y),
    \qquad y\in I_x^\circ .
\end{equation}
Thus, \(\Phi_x'\) and \(\Theta_x\) have the same sign.

\begin{Prop}\label{prop_decrease_of_phi}
    Fix \(x\ge0\), and suppose that \(k_x'(y)>0\) for all \(y\in I_x^\circ\). Assume that \(\Theta_x\) is convex on \(I_x^-\), and that $\Theta_x(y) \ge 0$ holds for every $y \in I_x^+$.
    Then the set
    \(
        \{y\in I_x^\circ:\Phi_x'(y)<0\}
    \)
    is either empty or a single interval.
\end{Prop}

\begin{proof}
    By \eqref{phi_derivative_theta}, we have
    \(
        \{y\in I_x^\circ:\Phi_x'(y)<0\}
        =
        \{y\in I_x^\circ:\Theta_x(y)<0\}.
    \)
    Since \(\Theta_x\ge0\) on \(I_x^+\), the latter set is contained in \(I_x^-\). On \(I_x^-\), the function \(\Theta_x\) is convex by assumption, and therefore its strict sublevel set
    \(
        \{y\in I_x^-:\Theta_x(y)<0\}
    \)
    is convex. Hence, this set is either empty or an interval.
\end{proof}

We now specialize this criterion to the standing assumptions on the terminal cost \(k(\cdot)\). Recall
that
\(
    k_x(y)=k(x+y),
    \, y>-x,
\)
so that
\[
    P_x(y)
    =
    \frac{k''(x+y)+2c}{2k'(x+y)}.
\]

\begin{Cor}\label{cor_phi_sufficient_condition}
    Fix \(x\ge0\), and assume that the cost function \(k(\cdot)\) satisfies the conditions \ref{ass_1}, \ref{ass_2} of Section \ref{sec:model}. Then the set
    \(
        \{y\in(-x,\infty):\Phi_x'(y)<0\}
    \)
    is either empty or a single interval, and the hypotheses of Theorem \ref{th_fixed_x_solution} are satisfied.
\end{Cor}

\begin{proof}
    Assumption \ref{ass_2} gives \(k'(s)>0\) for all \(s>0\), and therefore \(k_x'(y)>0\) for all \(y>-x\). It also says precisely that \(s\mapsto (k''(s)+2c)/(2k'(s))\) is convex on \((0,\infty)\). Hence, the function $P_x$ in \eqref{def_Qx_Thetax} is convex on \(I_x^\circ\), and in particular on \(I_x^-\).

    On the interval \(I_x^-\), the function \(G\) is strictly convex. Thus, \(\Theta_x=P_x+G\) is convex on \(I_x^-\). On the other hand, on \(I_x^+\) we have \(G\ge0\). Since \(k(\cdot)\) is convex and \(c>0\), we have \(P_x(y) > 0\) for \(y\in I_x^\circ\) in \eqref{def_Qx_Thetax}.
    Therefore, \(\Theta_x\ge0\) on \(I_x^+\), and Proposition \ref{prop_decrease_of_phi} applies.

    If the set of decrease is empty, then \(\Phi_x\) is non-decreasing and the first part of Theorem \ref{th_fixed_x_solution} applies. If it is nonempty, then it is a single interval; moreover it is contained in \(I_x^-\subset(-x,y_0)\), and hence its right endpoint is finite. This is exactly the one-interval geometry \eqref{one_interval_phi_geometry}.
\end{proof}

\section{Back to Two Dimensions: Structural Properties}\label{sec:structure}

\subsection{Notation and Overview}\label{subsec:2d_notation_overview}

In the previous section, we solved the embedded problem \eqref{2d_value_func_def} after fixing the first coordinate \(x\ge0\). We take up now the full two-dimensional formulation. The key point is that the diffusion \(Y_y(\cdot)\) evolves only in the \(y\)-coordinate, while \(x\) enters the problem as a parameter through the terminal cost \(y\mapsto k(x+y)\). Thus, the two-dimensional problem is obtained by gluing together the fixed-\(x\) solutions of Theorem~\ref{th_fixed_x_solution}.

The purpose of this section is to understand the geometry of this gluing; corresponding illustrations are provided in Section~\ref{sec:examples}. We shall see that the continuation region is a strip in the $(x, y)$ plane, whose vertical sections are the intervals found in the previous section. The main questions are then how this strip is generated, how its two boundaries move as \(x\) changes, and whether the lower boundary ever touches the diagonal \(y=-x\), the lower edge of the state space. These structural properties will be used in Section~\ref{sec:original_problem} to recover the solution of the original one-dimensional problem \eqref{value_function_def} from the horizontal section \(y=0\). 

We start by introducing relevant notions and notation.
The value function \(U(\cdot,\cdot)\) is defined in \eqref{2d_value_func_def}.
We denote the natural state space of the two-dimensional problem by
\begin{equation}\label{state_space_D}
    \mathcal D
    \coloneqq
    \{(x,y):x\ge0,\ y\ge -x\},
\end{equation}
and the continuation and stopping regions, respectively, by
\begin{equation}\label{def_C_S_2d}
    \mathcal C
    \coloneqq
    \{(x,y)\in\mathcal D:U(x,y)<k(x+y)\},
    \qquad 
    \mathcal S
    \coloneqq
    \{(x,y)\in\mathcal D:U(x,y)=k(x+y)\}.
\end{equation}
For fixed \(x\ge0\), the vertical section of \(\mathcal C\) is
\(
    \mathcal C_x
    \coloneqq
    \{y\ge -x:(x,y)\in\mathcal C\},
\)
as in \eqref{def_continuation_set}.
Throughout the present section, and unless stated otherwise, we assume that the hypotheses of Theorem~\ref{th_fixed_x_solution} hold for every \(x\ge0\). In particular, this is guaranteed by the sufficient condition in Corollary \ref{cor_phi_sufficient_condition}. Hence, each section \(\mathcal C_x\) is either empty or a single open interval. We define the set of parameters for which the continuation is nontrivial by
\begin{equation}\label{def_I_active}
    \mathcal I
    \coloneqq
    \{x\ge0:\mathcal C_x\neq\varnothing\}.
\end{equation}

We now switch from the fixed-\(x\) notation \(\Phi_x(y)\) of \eqref{def_phi} to the
bivariate notation
\begin{equation}\label{def_phi_2d}
    \Phi(x,y)
    \coloneqq
    \Phi_x(y)
    =
    \frac{k'(x+y)+cG(y)}{G'(y)}+cy,
    \qquad x\ge0,\quad y>-x.
\end{equation}
We shall denote partial derivatives of two-variable objects by $D_x$ and $D_y$, respectively, while primes will continue to denote derivatives of functions of one real variable. Whenever a value at the lower endpoint $y = -x$ is needed, we use the right limit
\begin{equation}\label{def_phi_boundary_2d}
    \Phi(x,-x)
    \coloneqq
    \frac{k'(0+)+cG(-x)}{G'(-x)}-cx.
\end{equation}
By Proposition~\ref{prop_Phi_prime} and Theorem~\ref{th_fixed_x_solution}, the region \eqref{def_I_active} is
\begin{equation}\label{def_I_via_phi}
    \mathcal I
    =
    \left\{
        x\ge0:
        \exists\, y>-x
        \text{ such that }
        D_y \Phi(x,y)<0
    \right\}.
\end{equation}

With the convention \(\inf\emptyset=\infty\), we denote by
\(
    x_{birth}
    \coloneqq
    \inf\mathcal I
\)
the first possible value of \(x\) for which a nonempty continuation section may appear.
For \(x\in\mathcal I\), we write
\(
    \mathcal C_x=(a(x),b(x)),
    \,
    -x\le a(x)<b(x)<\infty,
\)
and the continuation region becomes
\begin{equation}\label{C_sectional_representation}
    \mathcal C
    =
    \{(x,y):x\in\mathcal I,\ a(x)<y<b(x)\}.
\end{equation}
The functions \(a(\cdot)\) and \(b(\cdot)\) will be called \textit{lower} and \textit{upper free boundaries}.

The lower boundary has an additional feature, specific to the present two-dimensional embedding. Since the state space has lower edge \(y=-x\), the lower free boundary may touch this diagonal. We measure its distance from the diagonal by
\begin{equation}\label{def_d_lower_boundary}
    d(x)
    \coloneqq
    x+a(x),
    \qquad x\in\mathcal I,
\end{equation}
and define the first possible contact point
\(
x_{contact}
    \coloneqq
    \inf\{x\in\mathcal I:a(x)=-x\},
\)
again with the convention \(\inf\emptyset=\infty\).
Finally, for \(x\in\mathcal I\), we define the coefficient functions
\begin{equation}\label{def_q_R_coefficients}
    Q(x)
    \coloneqq
    \Phi(x,b(x)),
    \qquad
    R(x)
    \coloneqq
    k(x+a(x))-Q(x)G(a(x))+c\,a(x)G(a(x)).
\end{equation}
Then Theorem~\ref{th_fixed_x_solution} gives, on the continuation interval,
\begin{equation}\label{value_inside_C_coefficients}
    U(x,y)
    =
    Q(x)G(y)+R(x)-cyG(y),
    \qquad a(x)<y<b(x);
\end{equation}
and outside this interval, \(U(x,y)=k(x+y)\). If the lower boundary is interior, (i.e., \(a(x)>-x\)), then
\[
    Q(x)=\Phi(x,a(x))=\Phi(x,b(x)).
\]
If instead \(a(x)=-x\), the lower endpoint is pinned to the boundary of the state space, and the level \(Q(x)\) is determined by smooth fit only at the upper boundary.

The next subsection identifies the values of \(x\) for which nonempty continuation sections can appear. After that, we study the expansion and monotonicity properties of the continuation region, the free-boundary equations and their regularity, the possible contact of the lower boundary with the diagonal, the asymptotic behavior of the upper boundary, and finally the corresponding regularity and monotonicity properties of the value function.

The following result summarizes the structural conclusions of this section. The next subsections prove the individual assertions of this theorem and give the corresponding analytic formulas.

\begin{Th}\label{th_2d_structure}
    In the setting of Theorem~\ref{th_fixed_x_solution}, the two-dimensional continuation region $\mathcal{C}$ has the sectional representation
    \eqref{C_sectional_representation}, where each nonempty section \(\mathcal C_x\) is a bounded interval \((a(x),b(x))\). On this region the value function is given by
    \[
        U(x,y)=Q(x)G(y)+R(x)-cyG(y),
        \qquad a(x)<y<b(x),
    \]
    in the notation of \eqref{def_q_R_coefficients}, \eqref{def_G},
    while \(U(x,y)=k(x+y)\) outside \(\mathcal C\).

    \noindent \textbf{(i)} \quad 
    The continuation region expands diagonally: if \((x,y)\in\mathcal C\) and
    \(h\ge0\), then
    \(
        (x+h,y-h)\in\mathcal C.
    \)
    Consequently, \(\mathcal I\) of \eqref{def_I_active} is a right-interval, \(d(x)=x+a(x)\) is non-increasing on \(\mathcal I\), and diagonal contact is sticky: if
    \(a(\bar x)=-\bar x\) for some \(\bar x\in\mathcal I\), then
    \(
        a(x)=-x,\, x\in\mathcal I,\, x\ge\bar x.
    \)
    
    \noindent If, in addition, \ref{ass_3} holds, then the continuation region expands vertically:
    \[
        (x,y)\in\mathcal C,\ h\ge0
        \quad\Longrightarrow\quad
        (x+h,y)\in\mathcal C.
    \]
    In particular,
    \[
        a(x+h)\le a(x),
        \qquad
        b(x+h)\ge b(x),
        \qquad x\in\mathcal I,\ h\ge0.
    \]

    \noindent \textbf{(ii)} \quad 
    The free boundaries satisfy the integral equations \eqref{interior_boundary_system_2d}, \eqref{active_boundary_system_2d} below. If the relevant side branches of \(y\mapsto\Phi(x,y)\) have no flat portions, then the boundary functions are continuous on intervals where the boundary regime is fixed. At regular points, where the corresponding values of \(H\)  of \eqref{def_H_2d} do not vanish, the boundaries are of class \(C^1\) and satisfy the differential equations of Proposition~\ref{prop_boundary_regular_odes}.

    \noindent \textbf{(iii)} \quad 
    The value function is non-decreasing in both $x$ and $y$ variables. Moreover, it satisfies smooth fit in the displacement variable $y$ at the upper boundary and at every interior lower boundary. At regular boundary points it also satisfies horizontal smooth fit.

    \noindent \textbf{(iv)} \quad 
    Finally, if \(k'(0+)\le c\), then the lower boundary never touches the diagonal:
    \[
        a(x)>-x,\qquad x\in\mathcal I.
    \]
    If \(k'(0+)>c\) and \(k(r)=o(e^{2r})\) as \(r\to\infty\), then the lower
    boundary eventually coincides with the diagonal:
    \[
        a(x)=-x
        \quad\text{for all sufficiently large }x\in\mathcal I.
    \]
    In the non-contact case, under the same subcritical growth condition and assuming that
    \(
    \rho
    \coloneqq
    \sup\{r\ge0:k'_+(r)\le c\}
    \)
    is finite, one has
    \(
        x+a(x)\to\rho
    \)
    as $x\to\infty.$
\end{Th}

The theorem is intentionally stated in geometric terms. The precise analytic characterization of the birth point, the free-boundary equations, and the regularity statements, are given in the propositions below.

\subsection{Birth of a nonempty continuation region}\label{subsec:birth_continuation_sections}

We characterize now the set \(\mathcal I\) of \eqref{def_I_active} analytically. Recall from \eqref{phi_derivative} that the derivative of \(\Phi(x, \cdot)\) is controlled by the quantity \(\mathcal L k(x + \cdot)+c\). In bivariate notation, we define
\begin{equation}\label{def_H_2d}
    H(x,y)
    \coloneqq
    \frac12 k''(x+y)+G(y)k'(x+y)+c,
    \qquad x\ge0,\quad y>-x,
\end{equation}
and note
\begin{equation}\label{Phi_y_H_relation_2d}
    D_y \Phi(x,y)
    =
    \frac{2}{G'(y)}H(x,y),
    \qquad x\ge0,\quad y>-x.
\end{equation}
Since \(G'>0\), the sign of \(D_y \Phi\) is the same as the sign of \(H\).

\begin{Prop}\label{prop_birth_sections}
    Under the standing assumptions of this section,
    \(
        x\in\mathcal I
    \)
    if and only if
    \begin{equation}\label{criterion_nonempty_section}
        \inf_{s>0}
        \left\{
            \frac12 k''(s)+G(s-x)k'(s)+c
        \right\}<0.
    \end{equation}
    Equivalently, under the assumption \ref{ass_2} 
    \begin{equation}\label{criterion_nonempty_section_Q}
        x\in\mathcal I
        \quad\Longleftrightarrow\quad
        \exists\,s>0
        \text{ such that }
        G(s-x)<-P_k(s),
    \end{equation}
    with $P_k(\cdot)$ as in \eqref{Qk_def}
    Consequently, and with the convention \(\inf\emptyset=\infty\),
    \begin{equation}\label{x_under_formula_Q}
        x_{birth}
        \coloneqq \inf \mathcal{I}
        =
        \inf\left\{
            x\ge0:
            \exists\,s>0
            \text{ such that }
            G(s-x)<-P_k(s)
        \right\}.
    \end{equation}
\end{Prop}

\begin{proof}
    By \eqref{def_I_via_phi} and \eqref{Phi_y_H_relation_2d}, we have \(x\in\mathcal I\) if and only if there exists \(y>-x\) such that \(H(x,y)<0\). Writing \(s=x+y\), this is equivalent to the existence of \(s>0\) such that
    \(
        \frac{1}{2} k''(s)+G(s-x)k'(s)+c<0,
    \)
    which proves \eqref{criterion_nonempty_section}.
    Under \ref{ass_2}, we have \(k'(s)>0\) for \(s>0\), and the last inequality is equivalent to
    \(
    G(s-x) < -P_k(s)
    \)
    in the notation of \eqref{Qk_def}. This gives \eqref{criterion_nonempty_section_Q}. The expression \eqref{x_under_formula_Q} follows directly from the definition of
    \(x_{birth}\).
\end{proof}

The expression \eqref{x_under_formula_Q} is often the most convenient way to determine when nonempty continuation sections first appear. It also separates two different possibilities, which should not be confused: one may have \(x_{birth}=0\) because continuation appears immediately at \(x=0\), or because continuation appears only for arbitrarily small positive values of \(x\).

\begin{remark}[Special Functions]\label{remark_special_functions}
For functions of power form
\(
    k(s)=\alpha s^\beta+\gamma,
    \, s\ge0,
\)
with $\alpha>0,\, \beta\ge1$, the function of \eqref{def_Qx_Thetax}
becomes
\[
    P_k(s)
    =
    \frac{\beta-1}{2s}
    +
    \frac{c}{\alpha\beta}s^{1-\beta},
    \qquad s>0.
\]
If \(\beta>1\), then \(\inf_{s>0}P_k(s)=0\), and therefore \(\mathcal I\neq\varnothing\). If \(\beta=1\), then \(P_k(s)=c/\alpha\). Hence, nonempty continuation sections can occur if and only if \(c<\alpha\); in that case they appear for sufficiently large \(x\), while if
\(c\ge \alpha\), then \(\mathcal I=\varnothing\).

For exponential costs
\(
    k(s)=\alpha(e^{\lambda s}-\gamma),
    \, s\ge0,
\)
with $\alpha,\lambda>0$, the function of \eqref{def_Qx_Thetax} is
\[
    P_k(s)
    =
    \frac{\lambda}{2}
    +
    \frac{c}{\alpha\lambda}e^{-\lambda s},
    \qquad s>0.
\]
Thus, \(\inf_{s>0}P_k(s)=\lambda/2\). Hence, nonempty continuation sections can occur only in the subcritical regime \(\lambda<2\). In the subcritical regime, \(\mathcal I\neq\varnothing\); if \(\lambda\ge2\), then \(\mathcal I=\varnothing\), and immediate stopping is optimal in every fixed-\(x\) section. See Section \ref{sec:examples} for further discussion and illustrations.
\end{remark}

\subsection{Expansion properties of the continuation region}
\label{subsec:expansion_properties}

We record now the main global geometric properties of the continuation region.

The first is a diagonal expansion. It says that if it is optimal to continue from a point \((x,y)\), then it is also optimal to continue after increasing the first coordinate by \(h\) and decreasing the displacement by the same amount. This transformation keeps the terminal argument \(x+y\) unchanged, but makes the posterior drift more favorable. It is the geometric reason why the lower boundary can touch the diagonal \(y=-x\) at most once.

\begin{Prop}\label{prop_diagonal_expansion}
    If \((x,y)\in\mathcal C\) and \(h\ge0\), then
    \(
        (x+h,y-h)\in\mathcal C.
    \)
    Consequently, if \(\mathcal I\neq\varnothing\), then \(\mathcal I\) is a right interval $(x_{birth},\infty)$ or $[x_{birth},\infty)$. Moreover, for every \(x\in\mathcal I\) and every \(h\ge0\),
    \begin{equation}\label{diag_expansion_boundaries}
        a(x+h)\le a(x)-h,
        \qquad
        b(x+h)\ge b(x)-h.
    \end{equation}
    In particular, the distance-from-the-diagonal function $d(\cdot)$ of \eqref{def_d_lower_boundary} is non-increasing on \(\mathcal I\). Hence, if \(a(\bar x)=-\bar x\) for some \(\bar x\in\mathcal I\), then we have
    \(
        a(x)=-x
    \)
    for all
    \(
        x\in\mathcal I,\, x\ge\bar x.
    \)
\end{Prop}

\begin{proof}
    Let \((x,y)\in\mathcal C\). By definition, there exists a stopping time \(\tau\) such that
    \begin{equation}\label{expect_st_region}
        \ex\big[k(x+Y_y(\tau))+c\tau\big]<k(x+y).
    \end{equation}
    As in the fixed-\(x\) problem, replacing \(\tau\) by \(\tau\wedge\tau_{-x}\) cannot increase the cost. Thus, we may assume that
    \(\tau\le\tau_{-x}\).
    Now let \(\widetilde Y(\cdot)\) be the solution starting from \(y-h\), driven by the same Brownian motion as \(Y_y(\cdot)\), and set
    \[
        Z(t)\coloneqq \widetilde Y(t)+h.
    \]
    Then we note \(Z(0)=y=Y_y(0)\), and
    \(
        dZ(t)=G(Z(t)-h)\,dt+dN(t).
    \)
    Since \(G\) is increasing, \(G(z-h)\le G(z)\) for all \(z\). By the comparison theorem for one-dimensional diffusions (see, e.g., \cite[Proposition 5.2.18 on p. 293]{BMSC}), we then obtain
    \(
    Z(t)\le Y_y(t)
    \)
    for all $t\ge0$.

    Let
    \[
        \widetilde\tau_{-(x+h)}
        \coloneqq
        \inf\{t\ge0:\widetilde Y(t)\le -(x+h)\},
        \qquad
        \sigma\coloneqq \tau\wedge\widetilde\tau_{-(x+h)}.
    \]
    The stopping time \(\sigma\) is admissible for the problem started from \((x+h,y-h)\). If \(\sigma<\tau\), then \(x+h+\widetilde Y(\sigma)=0\), and the terminal cost is minimal. If
    \(\sigma=\tau\), then
    \[
        x+h+\widetilde Y(\tau)
        =
        x+Z(\tau)
        \le
        x+Y_y(\tau),
    \]
    and both sides are nonnegative. Since \(k\) is non-decreasing on
    \([0,\infty)\), we have in all cases
    \[
        k(x+h+\widetilde Y(\sigma))
        \le
        k(x+Y_y(\tau)).
    \]
    Since we have also \(\sigma\le\tau\), we deduce from \eqref{expect_st_region}
    \[
        \ex\big[k(x+h+\widetilde Y(\sigma))+c\sigma\big]
        \le
        \ex\big[k(x+Y_y(\tau))+c\tau\big]
        <
        k(x+y)
        =
        k(x+h+y-h).
    \]
    Hence, immediate stopping is not optimal at \((x+h,y-h)\), and so \((x+h,y-h)\in\mathcal C\).

    The set-theoretic consequences follow from the interval structure of the vertical sections. Indeed, if \(\mathcal C_x=(a(x),b(x))\), then the diagonal expansion gives
    \[
        (a(x)-h,b(x)-h)\subseteq \mathcal C_{x+h}.
    \]
    Hence \(x+h\in\mathcal I\) and
    \[
        a(x+h)\le a(x)-h,
        \qquad
        b(x+h)\ge b(x)-h.
    \]
    This proves \eqref{diag_expansion_boundaries} and shows that \(\mathcal I\) is a right interval. Finally, we note
    \[
        d(x+h)=x+h+a(x+h)\le x+a(x)=d(x),
    \]
    so the function \(d(\cdot)\) of \eqref{def_d_lower_boundary} is non-increasing. Since \(d(x)\ge0\), if \(d(\bar x)=0\), then \(0\le d(x)\le d(\bar x)=0\) for all \(x\ge\bar x\) in \(\mathcal I\). Thus, \(a(x)=-x\) for all such \(x\).
\end{proof}

The second expansion property of the continuation region is vertical. It is stronger than the diagonal one and will be used later in Theorem \ref{th_original_threshold} to show that the original one-dimensional problem is governed by a single threshold. This property relies on the log-concavity assumption \ref{ass_3}. We first isolate an elementary consequence of log-concavity which is needed in the proof.

\begin{Lem}\label{lem_translation_deficit}
    Assume \ref{ass_3}. Fix a constant \(h\ge0\). If \(X\ge0\) is a random variable, and \(u\ge0\), \(r\ge0\) are real constants such that
    \(
        \ex[k(X)]+r<k(u),
    \)
    then
    \[
        \ex[k(X+h)]+r<k(u+h).
    \]
\end{Lem}

\begin{proof}
    Since \(k'(\cdot)>0\) on \((0,\infty)\), the restriction of \(k(\cdot)\) to
    \([0,\infty)\) is strictly increasing. For fixed \(h\ge0\), we claim that the function
    \[
        T_h(r)\coloneqq k\big(k^{-1}(r)+h\big),
        \qquad r\in k([0,\infty)).
    \]
    is concave. Indeed, writing \(r=k(u)\), a direct
    calculation and the assumption \ref{ass_3} give
    \[
        T_h''(r)
        =
        \frac{k'(u+h)}{(k'(u))^2}
        \left(
            \frac{k''(u+h)}{k'(u+h)}
            -
            \frac{k''(u)}{k'(u)}
        \right)
        \le0.
    \]

    We set now
    \[
        v\coloneqq k^{-1}\big(\ex[k(X)]\big),
    \]
    and note that the strict inequality in the statement implies \(v<u\). Moreover, Jensen's
    inequality applied to the concave function \(T_h\), gives
    \[
        \ex[k(X+h)]
        =
        \ex\big[T_h(k(X))\big]
        \le
        T_h\big(\ex[k(X)]\big)
        =
        k(v+h).
    \]
    Since \(k(\cdot)\) is convex, the function
    \(h\mapsto k(u+h)-k(v+h)\) is non-decreasing whenever \(u>v\), and
    \[
        k(u+h)-k(v+h)\ge k(u)-k(v)
        =
        k(u)-\ex[k(X)].
    \]
    Combining the two inequalities, we deduce
    \(
        k(u+h)-\ex[k(X+h)]
        \ge
        k(u)-\ex[k(X)]
        >
        r,
    \)
    i.e., the desired claim.
\end{proof}

\begin{Prop}\label{prop_vertical_expansion}
    Assume \ref{ass_3}. If \((x,y)\in\mathcal C\) and \(h\ge0\), then
    \(
        (x+h,y)\in\mathcal C.
    \)
    Consequently, the upper boundary $b(\cdot)$ is non-decreasing on $\mathcal{I}$: for every \(x\in\mathcal I\) and every \(h\ge0\), we have
    \begin{equation}\label{vertical_expansion_boundaries}
        a(x+h)\le a(x),
        \qquad
        b(x+h)\ge b(x).
    \end{equation}
\end{Prop}

\begin{proof}
    Let \((x,y)\in\mathcal C\). As above, there exists a stopping time
    \(\tau\le\tau_{-x}\) such that
    \(
        \ex\big[k(x+Y_y(\tau))+c\tau\big]<k(x+y).
    \)
    Set
    \(
        X\coloneqq x+Y_y(\tau),
        \,
        u\coloneqq x+y.
    \)
    Since \(\tau\le\tau_{-x}\), we have \(X\ge0\) a.s. Lemma
    \ref{lem_translation_deficit} gives then, for any $h \ge 0$,
    \(
        \ex\big[k(X+h)\big]+c\,\ex[\tau]<k(u+h).
    \)
    Equivalently,
    \[
        \ex\big[k(x+h+Y_y(\tau))+c\tau\big]<k(x+h+y).
    \]
    The same stopping time \(\tau\) is admissible for the problem started from
    \((x+h,y)\).
    Hence, immediate stopping is not optimal at \((x+h,y)\), and so
    \((x+h,y)\in\mathcal C\).
    The boundary inequalities and the monotonicity of $b(\cdot)$ follow now readily.
\end{proof}

\subsection{Free-boundaries: their characterization and regularity}
\label{subsec:free_boundary_equations}

The expansion result of the previous section shows that the lower boundary can touch the diagonal \(y=-x\) at most at one point: once \(d(x)=x+a(x)\) becomes zero, it remains zero to the right. Thus, the two-dimensional free-boundary problem has two natural regimes. At points where \(d(x)>0\), both endpoints of the section \(\mathcal C_x=(a(x),b(x))\) are genuine free boundaries. At points where \(d(x)=0\), the lower endpoint is pinned to the boundary of the state space, and only the upper endpoint remains free.

We present now the integral equations satisfied by the boundaries $a(\cdot), b(\cdot)$. The next result is simply the two-dimensional restatement of the fixed-\(x\) construction in Theorem~\ref{th_fixed_x_solution}. For each fixed \(x\in\mathcal I\), the continuation section is the interval \((a(x),b(x))\), and the balancing condition \eqref{eq_main_balancing_condition} from Section~\ref{sec:one_d_problem} leads to the corresponding free-boundary equations \eqref{interior_boundary_system_2d} below. The only distinction is whether the lower endpoint is interior or pinned to the boundary of the state-space.

\begin{Prop}[Free-Boundary Integral Equations]\label{prop_free_boundary_equations}
    Recall the notation of \eqref{def_phi_2d}, and let \(x\in\mathcal I\). If \(a(x)>-x\), then the free-boundaries $a(\cdot), b(\cdot)$ satisfy
    \begin{equation}\label{interior_boundary_system_2d}
        \Phi(x,a(x))=\Phi(x,b(x)),
        \qquad
        \int_{a(x)}^{b(x)}
            \big(\Phi(x,z)-\Phi(x,a(x))\big)\,dG(z)=0.
    \end{equation}
    If \(a(x)=-x\), then the free-boundary $b(\cdot)$ satisfies
    \begin{equation}\label{active_boundary_system_2d}
        \int_{-x}^{b(x)}
            \big(\Phi(x,z)-\Phi(x,b(x))\big)\,dG(z)=0.
    \end{equation}
\end{Prop}

The expansion properties from the previous section stated in Propositions \ref{prop_diagonal_expansion}, \ref{prop_vertical_expansion} provide information about the boundaries' global monotonicity. In particular, \(d(x)=x+a(x)\) is non-increasing on \(\mathcal I\); whereas, under the additional log-concavity assumption \ref{ass_3}, Proposition \ref{prop_vertical_expansion} also gives the monotonicity of $b(\cdot)$. Thus, under \ref{ass_3}, the boundaries have one-sided limits at every point of \(\mathcal I\), and the only possible discontinuities are jumps. 

The next proposition shows that such jumps can occur only if the relevant branches of \(y\mapsto\Phi(x,y)\) contain flat portions.

\begin{Prop}\label{prop_boundary_continuity_no_flats}
    Let \(J\subset\mathcal I\) be an interval on which the boundary regime is fixed. Suppose that, for every \(x\in J\), the relevant side branches of \(y\mapsto\Phi(x,y)\) in the fixed-\(x\) construction are strictly monotone: the right branch, containing \(b(x)\), is strictly increasing, and, whenever \(a(x)>-x\), the left branch, containing \(a(x)\), is strictly increasing. 
    
    Then the boundary functions are continuous on \(J\). More explicitly, if \(a(x)>-x\) on \(J\), then both \(a(\cdot)\) and \(b(\cdot)\) are continuous on \(J\). If \(a(x)=-x\) on \(J\), then \(b(\cdot)\) is continuous on \(J\), while \(a(x)=-x\) is continuous automatically.
\end{Prop}

\begin{proof}
    Fix \(x_0\in J\), and let \(x_n\to x_0\). The endpoints \(a(x_n)\) and \(b(x_n)\) are locally bounded: the lower endpoints lie above \(-x_n\), while the upper endpoints are bounded locally by the a priori large-\(y\) stopping estimate, which can be chosen uniformly for \(x\) in a compact neighborhood of \(x_0\). Hence, after passing to a subsequence, we may assume
    \(
        a(x_n)\to a_*
    \)
    and 
    \(
        b(x_n)\to b_*.
    \)
    Moreover, we have $a_* < b_*$. Indeed, 
    since \(x_0\in\mathcal I\), \eqref{def_I_via_phi} and the
    continuity of \(D_y\Phi\) imply the existence of two real numbers \(-x_0<u<v\) and a neighborhood \(J_0\) of \(x_0\) such that
    \(
        D_y\Phi(x,y)<0
    \)
    for $x\in J_0,\, u\le y\le v.$
    Lemma \ref{lem_decrease_points_in_continuation} therefore gives
    \(
        [u,v]\subset\mathcal C_x,
        \, x\in J_0.
    \)
    Hence, for all sufficiently large \(n\),
    \(
        a(x_n)<u<v<b(x_n),
    \)
    and consequently
    \(
        a_*\le u<v\le b_*.
    \)

    Passing to the limit in the free-boundary equations of Proposition \ref{prop_free_boundary_equations}, using the continuity of \(\Phi\) and \(G\), and recalling that \(a_*<b_*\), shows that the limiting endpoints solve the same non-degenerate fixed-\(x_0\) balancing problem as the endpoints \(a(x_0),b(x_0)\).

    By the uniqueness of the balancing level in the fixed-\(x\) construction, the limiting smooth-fit level must be the level corresponding to \(x_0\). Since the relevant branches of \(y\mapsto\Phi(x_0,y)\) are strictly monotone, this level determines the endpoint on each branch uniquely. Therefore,
    \[
        a_*=a(x_0),\qquad b_*=b(x_0)
    \]
    in the interior regime, and \(b_*=b(x_0)\) in the diagonal-contact regime. Since every convergent subsequence has the same limit, the boundaries are continuous at \(x_0\).
\end{proof}

Under an additional non-degeneracy assumption on $\Phi$, we can improve continuity to continuous differentiability. Indeed, the integral equations of Proposition \ref{prop_free_boundary_equations} hold at every \(x\in\mathcal I\). To obtain regularity of $a(\cdot)$ and $b(\cdot)$, one can then apply the implicit function theorem. However, differentiating the integrals requires a local regularity condition, namely, non-degeneracy of the $y$-derivative of $\Phi$. We recall \eqref{def_H_2d}, \eqref{Phi_y_H_relation_2d}, and note that since \(G'>0\), the condition \(H(x,y)\neq0\) is precisely the requirement that the corresponding level curve of \(\Phi\) be nondegenerate in the \(y\)-direction.

\begin{Prop}[Differential Equations for the Free Boundaries]\label{prop_boundary_regular_odes}
    Let \(x_0\in\mathcal I\).
    
    \noindent \textbf{(i)} \quad Suppose first that \(a(x_0)>-x_0\) and
    \(
        H(x_0,a(x_0))H(x_0,b(x_0))\neq0.
    \)
    Then there exists a neighborhood \(J\) of \(x_0\) in \(\mathcal I\) such that \(a(\cdot)\) and \(b(\cdot)\) are \(C^1\) on \(J\). Moreover, with
    \begin{equation}\label{def_M_boundary}
        M(x; a, b)
        \coloneqq
        \frac{k'(x+b)-k'(x+a)}
             {G(b)-G(a)},
    \end{equation}
    we have, for \(x\in J\), the system of coupled ordinary differential equations
    \begin{equation}\label{ode_a_interior}
    \begin{split}
        a'(x)
        &=
        \frac{G'(a(x))}{2H(x,a(x))}
        \left(
            M\big(x; a(x), b(x)\big)-\frac{k''(x+a(x))}{G'(a(x))}
        \right),
        \\
        b'(x)
        &=
        \frac{G'(b(x))}{2H(x,b(x))}
        \left(
            M\big(x; a(x), b(x)\big)-\frac{k''(x+b(x))}{G'(b(x))}
        \right).       
    \end{split}
    \end{equation}

    \noindent \textbf{(ii)} \quad Suppose next that \(J\subset\mathcal I\) is an open interval on which \(a(x)=-x\), and that
    \(
        H(x,b(x))\neq0,\, x\in J.
    \)
    Then \(b(\cdot)\) is \(C^1\) on \(J\). Moreover, with
    \begin{equation}\label{def_M0_boundary}
        M_0(x; b)
        \coloneqq
        \frac{
            k'(x+b)-k'(0+)
            +
            \big(\Phi(x,-x)-\Phi(x,b)\big)G'(-x)
        }
        {G(b)-G(-x)},
    \end{equation}
    we have, for \(x\in J\), the differential equation
    \begin{equation}\label{ode_b_active}
        b'(x)
        =
        \frac{G'(b(x))}{2H(x,b(x))}
        \left(
            M_0(x; b(x))-\frac{k''(x+b(x))}{G'(b(x))}
        \right), \qquad x \in J
    \end{equation}
\end{Prop}

\begin{proof}
    We consider first the interior regime. Define
    \[
        F_1(x; a,b)\coloneqq \Phi(x,a)-\Phi(x,b),
    \qquad
        F_2(x; a,b)
        \coloneqq
        \int_a^b\big(\Phi(x,z)-\Phi(x,a)\big)\,dG(z).
    \]
    The boundary equations \eqref{interior_boundary_system_2d} are
    \(F_1=F_2=0\). At a solution, the Jacobian with respect to \((a,b)\) is
    \[
        \det
        \begin{pmatrix}
            D_y \Phi(x,a) & -D_y \Phi(x,b)\\
            -D_y \Phi(x,a)(G(b)-G(a)) & 0
        \end{pmatrix}
        =
        -D_y \Phi(x,a)D_y \Phi(x,b)(G(b)-G(a)).
    \]
    Since \(G\) is strictly increasing and \(D_y \Phi=2H/G'\) as in \eqref{Phi_y_H_relation_2d}, this determinant is nonzero under the nondegeneracy assumption. The implicit function theorem then gives a local \(C^1\) branch \((\tilde a(x),\tilde b(x))\) solving the boundary equations near \((a(x_0),b(x_0))\).
    For \(x\) close to \(x_0\), this branch remains on the same side branches of \(y\mapsto\Phi(x,y)\) as the original endpoints. Hence, it is an admissible fixed-\(x\) construction, and the uniqueness of the balancing level in Lemma~\ref{lem_uniqueness} identifies it with the actual boundaries \((a(x),b(x))\).

    To derive the differential equations in \eqref{ode_a_interior}, we differentiate
    \(F_2(x; a(x),b(x))=0\). The boundary terms cancel because
    \(\Phi(x,a(x))=\Phi(x,b(x))\), thus
    \[
        \frac{d}{dx}\Phi(x,a(x))
        =
        \frac{\int_{a(x)}^{b(x)}D_x \Phi(x,z)\,dG(z)}
             {G(b(x))-G(a(x))}.
    \]
    Since
    \(
        D_x \Phi(x,z)=\dfrac{k''(x+z)}{G'(z)},
    \)
    we obtain
    \[
        \int_{a(x)}^{b(x)}D_x \Phi(x,z)\,dG(z)
        =
        \int_{a(x)}^{b(x)}k''(x+z)\,dz
        =
        k'(x+b(x))-k'(x+a(x))
    \]
    and
    \(
        \frac{d}{dx} \Phi(x,a(x))=M\big(x; a(x), b(x)\big).
    \)
    Differentiating \(\Phi(x,a(x))=\Phi(x,b(x))\) gives also
    \(
        \frac{d}{dx} \Phi(x,b(x))=M(x; a(x), b(x)).
    \)
    Finally,
    \[
        \frac{d}{dx} \Phi(x,a(x))
        =
        \frac{k''(x+a(x))}{G'(a(x))}
        +
        \frac{2H(x,a(x))}{G'(a(x))}a'(x),
    \]
    and solving for \(a'(x)\) gives the first equation of \eqref{ode_a_interior}. The proof of the second equation is identical.

    We consider now the active-boundary regime \(a(x)=-x\). Define
    \[
        F_0(x; b)
        \coloneqq
        \int_{-x}^{b}\big(\Phi(x,z)-\Phi(x,b)\big)\,dG(z).
    \]
    The equation \eqref{active_boundary_system_2d} is \(F_0(x,b(x))=0\). At a solution, the derivative
    \(
        D_b F_0(x,b)
        =
        -D_y \Phi(x,b)\big(G(b)-G(-x)\big)
    \)
    is nonzero under the stated nondegeneracy assumption. The implicit function theorem gives \(C^1\)-regularity of \(b(\cdot)\) on \(J\).
    Differentiating \(F_0(x,b(x))=0\), taking into account the moving lower endpoint \(-x\), yields
    \[
        \frac{d}{dx} \Phi(x,b(x))
        =
        \frac{
            k'(x+b(x))-k'(0+)
            +
            \big(\Phi(x,-x)-\Phi(x,b(x))\big)G'(-x)
        }
        {G(b(x))-G(-x)}
        =
        M_0\big(x; b(x)\big).
    \]
    Since
    \[
        \frac{d}{dx} \Phi(x,b(x))
        =
        \frac{k''(x+b(x))}{G'(b(x))}
        +
        \frac{2H(x,b(x))}{G'(b(x))}b'(x),
    \]
    solving for \(b'(x)\) gives the differential equation \eqref{ode_b_active}.
\end{proof}
\vspace{10pt}

\begin{remark}
\begingroup
\setlength{\parskip}{0pt}
    For the main examples of Remark \ref{remark_special_functions}, the additional assumptions of the preceding propositions hold. Direct computations for power costs \(k(s)=\alpha s^\beta+\gamma\), \(\beta\ge1\), and exponential costs \(k(s)=\alpha (e^{\lambda s}-\gamma), \lambda > 0\), show that the side branches of \(y\mapsto\Phi(x,y)\) are strictly monotone whenever the continuation section is nonempty. Moreover, we have \(H(x,a(x))>0\) at every interior lower boundary, and \(H(x,b(x))>0\) at every upper boundary. 

    Thus, for these examples, the free boundaries are \(C^1\) on each regime: before diagonal contact, both \(a\) and \(b\) are \(C^1\); after diagonal contact, \(a(x)=-x\) and \(b\) is \(C^1\). The only possible loss of global \(C^1\)-regularity occurs at the transition point where the lower boundary first hits the diagonal \(y=-x\) \emph{and}, if the birth point belongs to \(\mathcal I\), at the birth of the continuation interval.
\endgroup
\end{remark}

\subsection{The lower boundary and the diagonal}
\label{subsec:lower_boundary_diagonal}

We study now the asymptotic behavior of the lower boundary $a(\cdot)$. We have already shown that this boundary is continuous and differentiable under the mild regularity assumptions of the previous section. Moreover, the diagonal expansion property of Proposition \ref{prop_diagonal_expansion} shows that $a(\cdot)$ tends to the diagonal $y = -x + \alpha$ for some $\alpha \ge 0$; and, if it ever touches the diagonal, remains glued to it for all larger values of \(x\). Therefore, two natural questions remain. Does the lower boundary ever hit the diagonal? If it does not, how far above the diagonal does it stay asymptotically?

It is useful to express the vertical sections in the terminal-coordinate variable
\(
    s=x+y,
\)
which is the argument of the terminal cost. For fixed \(x\), the state space \(y\ge -x\) becomes \(s\ge0\), and the diagonal \(y=-x\) becomes the point \(s=0\). Moreover, the vertical section
\(
    \mathcal C_x=(a(x),b(x))
\)
of the continuation region $\mathcal C$
is transformed into
\(
    (x+a(x),\,x+b(x)).
\)
Thus,
\begin{equation}\label{def_d_boundary}
    d(x)=x+a(x)
\end{equation}
is the lower endpoint of the continuation section in terminal-coordinate variables. Proposition \ref{prop_diagonal_expansion} implies that this function is non-increasing. In particular, \(d(x)=0\) if and only if the lower boundary touches the diagonal. It remains to analyze $\lim_{x \to \infty} d(x)$.

Our main results below are based on the following idea. For fixed \(s\) and large \(x\), the corresponding displacement is \(y=s-x\), and the drift \(G(s-x)\) is close to \(-1\). Thus, near the diagonal and for large \(x\), the terminal-coordinate process is pushed strongly toward zero. The local comparison is then between the possible decrease of the terminal cost, governed by \(k'(s)\), and the running cost \(c\). If \(k'(s)\le c\), the possible saving from moving left is no larger than the running cost, and stopping is favorable. If \(k'(s)>c\), continuation is locally attractive, with the caveat that this can be spoiled by rare upward excursions before the process reaches the lower level. We make now this idea rigorous.

Our first statement makes the preceding ``small derivative implies stopping'' intuition precise by comparing the running cost $c$ with the best possible local improvement $k'(s) G(-x)$.

\begin{Lem}\label{lem_slope_stop}
    Fix \(x\ge0\) and \(s = x + y>0\). If
    \begin{equation}\label{slope_stop_condition}
        k'(s)G(-x)+c\ge0,
    \end{equation}
    then immediate stopping is optimal at the point \((x,s-x)\):
    \(
        U(x,s-x)=k(s).
    \)
    At \(s=0\), immediate stopping is optimal as well.
\end{Lem}

\begin{proof}
    The case \(s=0\) is immediate, since \(k\) is minimized at zero and the running cost is nonnegative.

    Thus, let \(s>0\), and let \(\tau\) be any admissible stopping time with finite expectation. As before, replacing \(\tau\) by its minimum with the first hitting time of \(0\) by the process \(x+Y_{s-x}(\cdot)\) cannot increase the cost. Thus, we may assume that
    \begin{equation}\label{def_Z_process}
        Z(t)\coloneqq x+Y_{s-x}(t)\ge0,\qquad 0\le t\le\tau .
    \end{equation}
    Since \(Y_{s-x}(t)\ge -x\) on this interval and \(G\) is increasing,
    \(
        G(Y_{s-x}(t))\ge G(-x),\, 0\le t\le\tau .
    \)
    By convexity,
    \(
        k(r)\ge k(s)+k'(s)(r-s),\, r\ge0.
    \)
    Therefore,
    \[
        \begin{split}
        \ex\big[k(Z(\tau))+c\tau\big]
        &\ge
        k(s)+k'(s)\ex[Z(\tau)-s]+c\,\ex[\tau] \\
        &=
        k(s)
        +k'(s)\ex\left[\int_0^\tau G(Y_{s-x}(u))\,du\right]
        +c\,\ex[\tau] \\
        &\ge
        k(s)+\big(k'(s)G(-x)+c\big)\ex[\tau]
        \ge k(s).
        \end{split}
    \]
    Here, to obtain the equality on the second line, we used Wald's identity $\ex[N(\tau)] = 0$ for a standard Brownian motion $N(\cdot)$ thanks to $\ex[\tau] < \infty$. As a result, no stopping time $\tau$ can improve immediate stopping, which attains the above lower bound \(k(s)\), and the result follows.
\end{proof}

Lemma \ref{lem_slope_stop} immediately rules out diagonal contact when the slope of \(k(\cdot)\) at the
origin is not larger than the running cost.

\begin{Prop}\label{prop_no_contact_m_le_c}
    If \(k'(0+)\le c\), then we have
    \(
        a(x)>-x
    \)
    for all $x\in\mathcal I.$
\end{Prop}

\begin{proof}
    Fix \(x\in\mathcal I\). We show that the continuation section
    \(\mathcal C_x\) cannot start at the diagonal.

    First, we note that $x \in \mathcal{I}$ implies \(G(-x) < 0\). Indeed, if \(G(-x)\ge0\), then \eqref{slope_stop_condition} holds for every \(s\ge0\), since \(k'\ge0\) on \((0,\infty)\). Lemma
    \ref{lem_slope_stop} therefore gives
    \(
        U(x,s-x)=k(s),\ s\ge0.
    \)
    Hence \(\mathcal C_x=\varnothing\), contradicting \(x\in\mathcal I\).

    Thus, we assume $G(-x) < 0$. In this case, since $G > -1$, we have 
    \[
        \frac{c}{-G(-x)}>c\ge k'(0+).
    \]
    Since \(k'(s)\to k'(0+)\) as \(s\downarrow0\), there exists
    \(\varepsilon>0\) such that
    \(
        k'(s)G(-x)+c\ge0,\, 0<s<\varepsilon.
    \)
    Lemma \ref{lem_slope_stop} then gives
    \[
        U(x,s-x)=k(s),\qquad 0<s<\varepsilon.
    \]
    Thus, points arbitrarily close to the diagonal belong to the stopping region, and the conclusion follows readily by using the fact that the continuation region is a connected interval.
\end{proof}

The converse direction is subtler. If \(k'(0+)>c\), then near the diagonal the first-order gain from reaching the minimum of \(k(\cdot)\) dominates the running cost. However, before hitting the diagonal, the terminal-coordinate process may make a rare upward excursion. The proof below controls this possibility by stopping the process at the first exit from \((0,R)\). The probability of hitting the upper level \(R\) before zero is of order \(e^{-2R}\), so the possible cost of this upward excursion is negligible under the subcritical growth condition
\begin{equation}\label{subcritical_growth_k}
    k(R)=o(e^{2R}),\qquad R\to\infty.
\end{equation}

\begin{Prop}\label{th_lower_boundary_dichotomy}
    Assume that \(\mathcal I\neq\varnothing\). 
    If \(k'(0+)>c\) and the growth condition \eqref{subcritical_growth_k} holds, then there exists
    \(x_0\in\mathcal I\) such that the diagonal contact condition $a(x) = -x$ holds for all $x\in\mathcal I,\, x\ge x_0.$
\end{Prop}

\begin{proof}
    As suggested by the preceding discussion, we shall show that, for \(x\) large and for terminal coordinates \(s=x+y\) close to zero, it is strictly better to wait until the terminal-coordinate process exits a fixed interval \((0,R)\) than to stop immediately. The level \(R\) will first be chosen large enough to make the rare upward excursion negligible; then \(x\) will be chosen large enough so that, on \([0,R]\), the drift \(G(s-x)\) is close to \(-1\).

    \underline{Step 1:} It is convenient to work in the terminal
    coordinate \(s=x+y\). For fixed \(x\ge0\) and \(s\ge0\), define
    \(
        Z_s^x(t)
        \coloneqq
        x+Y_{s-x}(t),
        \, t\ge0.
    \)
    Then \(Z_s^x(0)=s\), and
    \begin{equation}\label{sde_Z_terminal_coordinate}
        Z_s^x(t)
        =
        s+\int_0^t G\big(Z_s^x(u)-x\big)\,du+N(t).
    \end{equation}
    The point \(Z_s^x=0\) corresponds to the diagonal \(y=-x\).
    
    Fix \(R>0\), and consider the cost 
    \(
        J^x_R(s) \coloneqq \ex\big[k\big( Z^x_s(\tau_R^{x, s})\big) + c \, \tau_R^{x, s}\big]
    \)
    of stopping at the time
    \(
        \tau_R^{x,s}
        \coloneqq
        \inf\{t\ge0:Z_s^x(t)\notin (0,R)\},
        \, 0<s<R.
    \)
    We shall compare this \(J_R^x(s)\) with the cost \(k(s)\) of immediate stopping, for small \(s\). To do so, we shall find the right derivative of \(J_R^x\) at zero and show that, for large enough $R$ and $x$, it is strictly smaller than \(k'(0+)\), which will imply that $k(\cdot)$ dominates \(J_R^x\) at the neighborhood of the origin.

    \underline{Step 2:} We claim that
    \begin{equation}\label{derivative_limit_J}
        (J_R^x)'(0+)
        \longrightarrow
        c\left(1-\frac{2R}{e^{2R}-1}\right)
        +
        \frac{2\bigl(k(R)-k(0)\bigr)}{e^{2R}-1},
        \qquad x\to\infty.
    \end{equation}
    Assume first that this has been established. Note that, by the subcritical growth assumption \(k(R)=o(e^{2R})\) in \eqref{subcritical_growth_k}, the right-hand side converges to \(c\) as \(R\to\infty\). Therefore, since \(k'(0+)>c\), we can choose \(R\) so large that, for this fixed \(R\) and all sufficiently large \(x\), we have
    \(
        (J_R^x)'(0+)<k'(0+).
    \)
    Since \(J_R^x(0)=k(0)\), it follows that, for each such \(x\),
    \[
        \lim_{s\downarrow0}
        \frac{k(s)-J_R^x(s)}{s}
        =
        k'(0+)-(J_R^x)'(0+)
        >0.
    \]
    Consequently, for each sufficiently large \(x\), there exists \(\varepsilon_x>0\) such that
    \(
        J_R^x(s)<k(s),
        \, 0<s<\varepsilon_x.
    \)
    Thus, for every sufficiently large \(x\), every terminal coordinate \(s\in(0,\varepsilon_x)\) belongs to the continuation region. Equivalently, the continuation section \(\mathcal C_x\) contains points arbitrarily close to the diagonal, and hence
    \(
        a(x)=-x.
    \)
    By the stickiness of diagonal contact from Proposition \ref{prop_diagonal_expansion}, this remains true for all larger \(x\in\mathcal I\). It is therefore sufficient to establish \eqref{derivative_limit_J} in order to conclude the proof.

    \underline{Step 3:}
    Let
    \(
        p_R^x(s)
        \coloneqq
        \pr\big(Z_s^x(\tau_R^{x,s})=R\big),
        \,
        u_R^x(s)
        \coloneqq
        \ex[\tau_R^{x,s}],
    \)
    so that
    \[
        J_R^x(s)
        =
        k(0)\big(1-p_R^x(s)\big)+k(R)\,p_R^x(s)+c\,u_R^x(s),
    \]
    and we shall find analytic formulas for $p_R^x(\cdot)$ and $u_R^x(\cdot)$ starting with the former.
    
    As we have already discussed in the proof of Proposition \ref{prop_stopping_for_large_y}, the scale function of the process \(Z_s^x\) is given by \(\zeta\mapsto G(\zeta-x)\). Therefore, the probability of exiting from the interval $(0, R)$ through the right end-point is given by (see, e.g., \cite[Proposition 3.2 on p. 301]{RevuzYor})
    \[
        p_R^x(s)
        =
        \frac{G(s-x)-G(-x)}{G(R-x)-G(-x)};
    \]
    hence, as simple algebra shows, 
    \begin{equation}\label{p_derivative_limit}
        (p_R^x)'(0+)
        =
        \frac{G'(-x)}{G(R-x)-G(-x)}
        \longrightarrow
        \frac{2}{e^{2R}-1},
        \qquad x\to\infty.
    \end{equation}

    For \(u_R^x(\cdot)\), it is standard that this function solves the boundary
    value problem (see, e.g., \cite[Section 5.5.C on pp. 342--344]{BMSC})
    \[
        \frac12 u''(s)+G(s-x)u'(s)=-1,\qquad
        u(0)=u(R)=0.
    \]
    This equation has the same form as the continuation ODE \eqref{continuation_ode_fixed_x} in
    Section~\ref{sec:one_d_problem}, after the change of variables \(y=s-x\).
    Thus, its general solution is
    \[
        u^x(s)=\eta\,G(s-x)+\zeta-(s-x)G(s-x),
    \]
    for constants \(\eta,\zeta\). Imposing \(u^x(0)=u^x(R)=0\) gives
    \[
        \eta
        =
        \frac{(R-x)G(R-x)+xG(-x)}
             {G(R-x)-G(-x)}.
    \]
    Therefore,
    \(
        (u_R^x)'(s)
        =
        \eta G'(s-x)-G(s-x)-(s-x)G'(s-x),
    \)
    and hence
    \begin{equation}\label{u_derivative_explicit}
        (u_R^x)'(0+)
        =
        \frac{R\,G(R-x)G'(-x)}
             {G(R-x)-G(-x)}
        -
        G(-x).
    \end{equation}
    Now simple algebra gives
    \[
        \frac{G'(-x)}{G(R-x)-G(-x)}
        \longrightarrow
        \frac{2}{e^{2R}-1},
        \qquad x\to\infty,
    \]
    and together with \(G(R-x)\to-1\) and \(G(-x)\to-1\), this yields
    \begin{equation}\label{u_derivative_limit}
        (u_R^x)'(0+)
        \longrightarrow
        1-\frac{2R}{e^{2R}-1},
        \qquad x\to\infty.
    \end{equation}
    Combining \eqref{p_derivative_limit} and \eqref{u_derivative_limit}, we obtain
    \eqref{derivative_limit_J} and finish the proof.
\end{proof}

\begin{remark}\label{remark_required_extra_condition}
    It is not hard to see that the condition \(k'(0+)>c\) alone cannot imply diagonal contact. Indeed, consider exponential costs \(k(s)=\alpha(e^{\lambda s}-\gamma)\) and arbitrary running cost $c > 0$. One has \(k'(0+)=\alpha\lambda\), so it is possible that \(k'(0+)>c\) for large enough $\alpha$. However, if \(\lambda\ge2\), then \(\mathcal I=\varnothing\) by Proposition \ref{prop_birth_sections}.
    Thus, no continuation section exists; in particular, there is no lower boundary which could reach the diagonal. This explains why the converse direction in Proposition \ref{th_lower_boundary_dichotomy}     includes an additional subcritical-growth assumption.
\end{remark}

Finally, it remains to identify the asymptote of the lower boundary in the non-contact case. To do so, we can apply the same intuition as before. For large \(x\), the terminal-coordinate drift is close to \(-1\) on bounded \(s\)-intervals, so the local comparison is essentially between \(k'(s)\) and \(c\). Thus, the critical (asymptotic) level is where the derivative of \(k\) crosses the running cost level $c$. Since \(k\) is convex, \(k'\) is non-decreasing on \((0,\infty)\). We define
\begin{equation}\label{def_rho_asymptote}
    \rho
    \coloneqq
    \sup\{r>0:k'(r)\le c\}
\end{equation}
with the convention that \(\rho=0\) if the set is empty, and note that, when finite, \(\rho\) is the largest minimizer of the convex function \(r\mapsto k(r)-cr\). If \(\rho=\infty\), then \(k'(r)\le c\) for every \(r>0\), and Lemma \ref{lem_slope_stop} implies that immediate stopping is optimal everywhere, i.e., $\mathcal{I} = \emptyset$. Hence, in the only non-trivial case \(\mathcal I \neq \varnothing\) one necessarily has \(\rho<\infty\).

The next corollary has two parts. The lower bound \(d(x)\ge\rho\) uses only the stopping criterion above. The reverse bound uses the same exit argument as in the proof of Proposition \ref{th_lower_boundary_dichotomy}, now applied to the interval \((\rho,R)\).

\begin{Cor}\label{cor_lower_boundary_asymptote}
    Assume that \(\mathcal I\neq\varnothing\) and that \(a(x)>-x\) for every
    \(x\in\mathcal I\). Then
    \(
        d(x)\ge \rho
    \)
    for all $ x\in\mathcal I.$
    If, in addition, \eqref{subcritical_growth_k} holds, then
    \[
        \lim_{x\to\infty} d(x)=\rho, \qquad \text{thus} \qquad a(x)=-x+\rho+o(1) \qquad \text{as} \ x\to\infty.
    \]
\end{Cor}

\begin{proof}
    We first prove the lower bound. Let \(0<s<\rho\). By the definition of
    \(\rho\), we have \(k'(s)\le c\). Since \(G\ge -1\), we have
    \[
        k'(s)G(-x)+c\ge c-k'(s)\ge0.
    \]
    Lemma \ref{lem_slope_stop} implies that immediate stopping is optimal at the point \((x,s-x)\), for every \(x\ge0\). Thus, no point with terminal coordinate \(s<\rho\) can belong to the continuation region, and
    \(
        d(x)=x+a(x)\ge \rho,
    \)
    $x\in\mathcal I$.
    We now prove the reverse bound under \eqref{subcritical_growth_k}. Fix \(s>\rho\). Since \(\rho\) is the largest minimizer of
    \(r\mapsto k(r)-cr\), we have
    \(
        k(s)>k(\rho)+c(s-\rho).
    \)
    We show now that, for all sufficiently large \(x\), the point with terminal coordinate \(s\) belongs to the continuation region.

    Choosing \(R>s\), we let
    \[
        \tau_{\rho,R}^{x,s}
        \coloneqq
        \inf\{t\ge0:Z_s^x(t)\notin(\rho,R)\},
        \qquad
        J_{\rho,R}^x(s)
        \coloneqq
        \ex\big[k(Z_s^x(\tau_{\rho,R}^{x,s}))
        +c\,\tau_{\rho,R}^{x,s}\big].
    \]
    Exactly as in the proof of Proposition \ref{th_lower_boundary_dichotomy}, but now with the lower endpoint \(0\) replaced by \(\rho\), the scale formula and the explicit expected-exit-time solution give
    \[
        J_{\rho,R}^x(s)
        \longrightarrow
        k(\rho)\big(1-\bar p_{\rho,R}(s)\big)
        +k(R)\bar p_{\rho,R}(s)
        +c\,\bar u_{\rho,R}(s),
        \qquad x\to\infty,
    \]
    where
    \[
        \bar p_{\rho,R}(s)
        =
        \frac{e^{2(s-\rho)}-1}{e^{2(R-\rho)}-1} \quad \text{ and } \quad 
        \bar u_{\rho,R}(s)
        =
        (s-\rho)
        -(R-\rho)
        \frac{e^{2(s-\rho)}-1}{e^{2(R-\rho)}-1}.
    \]
    Letting \(R\to\infty\), we have
    \(
        \bar p_{\rho,R}(s)\to0
    \)
    and
    \(
        \bar u_{\rho,R}(s)\to s-\rho,
    \)
    and the subcritical growth condition \eqref{subcritical_growth_k} gives
    \(k(R)\bar p_{\rho,R}(s)\to0\). Therefore
    \[
        \lim_{R\to\infty}\lim_{x\to\infty}J_{\rho,R}^x(s)
        =
        k(\rho)+c(s-\rho).
    \]
    Since this is strictly smaller than \(k(s)\), we may choose \(R>s\) and then \(x\) large enough so that
    \[
        J_{\rho,R}^x(s)<k(s).
    \]
    Thus, waiting until \(\tau_{\rho,R}^{x,s}\) improves upon immediate stopping from \((x,s-x)\), and hence \(s-x\in\mathcal C_x\) for all sufficiently large \(x\). In particular, we obtain $d(x)=x+a(x)\le s$ for all sufficiently large \(x\), which finishes the proof upon recalling that \(s>\rho\) was arbitrary.
\end{proof}
\vspace{0.25\baselineskip}

\begin{remark}
\begingroup

\setlength{\parskip}{0pt}
    For power terminal costs
    \(
        k(s)=\alpha s^\beta+\gamma,
        \ \alpha>0, \beta>1,
    \)
    we have \(k'(0+)=0<c\). Hence the lower boundary never touches the diagonal,
    and
    \[
        \rho=\left(\frac{c}{\alpha\beta}\right)^{1/(\beta-1)}.
    \]
    In particular, for the normalized quadratic cost \(k(s)=s^2/2\), one has \(\rho=c\).

    For linear costs \(k(s)=\alpha s+\gamma\), one has \(k'(0+)=\alpha\). If \(\alpha\le c\), then immediate stopping is optimal everywhere. If \(\alpha>c\), then the lower boundary eventually coincides with the diagonal.

    For exponential costs
    \(
        k(s)=\alpha(e^{\lambda s}-\gamma),
        \ \alpha,\lambda>0,
    \)
    one has \(k'(0+)=\alpha\lambda\). In the subcritical regime \(\lambda<2\), the lower boundary never touches the diagonal if \(\alpha\lambda\le c\), while it eventually coincides with the diagonal if \(\alpha\lambda>c\). In the non-contact case,
    \[
        \rho=\frac{1}{\lambda}\log\left(\frac{c}{\alpha\lambda}\right).
    \]
\endgroup
\end{remark}

\subsection{Upper-boundary asymptotics}
\label{subsec:upper_boundary_asymptotics}

We now study the asymptotic behavior of the upper boundary \(b(\cdot)\). As in the previous subsection, it is useful to pass from the displacement variable \(y\)
to the terminal-coordinate variable
\(
    s=x+y.
\)
For the lower boundary, this led to the quantity
\(
    d(x)=x+a(x)
\)
of \eqref{def_d_boundary},
which measures the distance from the diagonal in terminal-coordinate variables. Importantly, we already obtained that $d(\cdot)$ is non-increasing, which, together with the non-negativity property, implies its boundedness.

For the upper boundary we shall use the analogous notation
\begin{equation}\label{def_e_upper_boundary}
    e(x)\coloneqq x+b(x).
\end{equation}
Thus, \(b(\cdot)\) is the upper endpoint of the continuation region in \(y\)-coordinates, while \(e(\cdot)\) is the same endpoint measured in the terminal-cost coordinate
\(s\).

The main point of this subsection is that the fixed-\(x\) balancing condition implies a simple asymptotic relation between these two quantities. At the upper boundary, the posterior weight of the negative drift must balance the relative marginal growth of the terminal cost. More precisely, under the tail condition \(k'(s)\to\infty\), we shall prove the fundamental balance
\begin{equation}\label{informal_upper_balance}
    1-G(b(x))
    \sim
    \frac{k'(e(x))}{k(e(x))}
\end{equation}
between the posterior probability
\((1-G(b(x)))/2\) of the negative drift at the upper boundary and the logarithmic derivative \(k'(e(x))/k(e(x))\) of the terminal cost at the corresponding terminal level.

This balance explains the possible regimes. If \(k'(s)/k(s)\to0\), then \(1-G(b(x))\to0\), and hence \(b(x)\to+\infty\). If \(k'(s)/k(s)\to\lambda\in(0,2)\), then the balance gives the finite limit
\[
    b(x)\to G^{-1}(1-\lambda)
    =
    y_0+\frac12\log\frac{2-\lambda}{\lambda}.
\]
The upper bound \(2\) is intrinsic: since \(G\in(-1,1)\), the quantity \(1-G\) takes values only in \((0,2)\). Thus, if \(k'(s)/k(s)\to\lambda\ge2\), the balance \eqref{informal_upper_balance} cannot hold along any tail on which \(b(\cdot)\) is bounded below. Importantly, under the vertical monotonicity assumption \ref{ass_3}, the boundary \(b(\cdot)\) is non-decreasing, so this bounded-below condition is automatic whenever continuation sections exist for all large \(x\). Hence, a supercritical tail \(\lambda\ge2\) rules out a nonempty continuation tail. This also agrees with Proposition \ref{prop_birth_sections} for exponential costs: if the exponential rate is at least \(2\), then \(\mathcal I=\varnothing\).

The exact identity \eqref{upper_endpoint_identity} below is the starting point. It rewrites the fixed-\(x\) balancing condition at the upper endpoint and is then reduced to \eqref{informal_upper_balance} by letting \(x\to\infty\).

\begin{Lem}\label{lem_upper_endpoint_identity}
    Let \(x\in\mathcal I\), and set
    \(
        a=a(x),\, b=b(x),\,
        d=x+a,\, e=x+b.
    \)
    Then
    \begin{equation}\label{upper_endpoint_identity}
        \frac{G'(b)}{G(b)-G(a)}
        \left[
            k(e)-k(d)+cG(a)(b-a)
        \right]
        =
        k'(e)+cG(b).
    \end{equation}
\end{Lem}

\begin{proof}
    In both boundary regimes, we know from Section \ref{subsec_main_construction} that the balancing equation is
    \begin{equation}\label{technical_balance}
        \int_a^b \big(\Phi(x,z)-\Phi(x,b)\big)\,dG(z)=0,  \quad \text{ thus also } \quad \int_a^b \Phi(x,z)\,dG(z)
        =
        \Phi(x,b)\big(G(b)-G(a)\big).
    \end{equation}
    Using the definition of \(\Phi\), the integral part becomes
    \[
        \begin{split}
        \int_a^b \Phi(x,z)\,dG(z) =
        \int_a^b
        \left[
            k'(x+z)+cG(z)+czG'(z)
        \right]dz =
        k(e)-k(d)+c\big(bG(b)-aG(a)\big).
        \end{split}
    \]
    On the other hand (cf. \eqref{def_phi_2d}),
    \[
        \Phi(x,b)=\frac{k'(e)+cG(b)}{G'(b)}+cb.
    \]
    Substituting this into the identity above and rearranging gives
    \eqref{upper_endpoint_identity}.
\end{proof}

The identity \eqref{upper_endpoint_identity} implies the asymptotic balance \eqref{informal_upper_balance}.

\begin{Prop}\label{prop_upper_boundary_balance}
    Assume that \(\mathcal I\neq\varnothing\), and that
    \(
        k'(s)\to\infty
    \)
    as $s\to\infty$.
    Then, along every sequence \(x\to\infty\), \(x\in\mathcal I\), for which
    \(b(x)\) is bounded below,
    \begin{equation}\label{upper_boundary_balance}
        1-G(b(x))
        =
        \frac{k'(x+b(x))}{k(x+b(x))}(1+o(1)).
    \end{equation}
\end{Prop}

\begin{proof}
    Let \(a=a(x)\), \(b=b(x)\), \(d=x+a\), and \(e=x+b\). Since
    \(d\) is non-negative and non-increasing, it is bounded along the tail of
    \(\mathcal I\). Since \(b\) is bounded below along the sequence under
    consideration, we have \(e\to\infty\), while \(a=d-x\to-\infty\), and hence
    \(G(a)\to-1\). Moreover, \(k(d)\) is bounded. Since \(k\) is convex and
    \(k'(s)\to\infty\), we have \(k(s)/s\to\infty\). As \(b-a=e-d\), this gives
    \(
        \big|cG(a)(b-a)\big|
        \le c(e+O(1))
        =
        o(k(e)).
    \)
    Consequently,
    \(
        k(e)-k(d)+cG(a)(b-a)
        =
        k(e)(1+o(1)),
    \)
    and since \(k'(e)\to\infty\), we also have
    \(
        k'(e)+cG(b)=k'(e)(1+o(1)).
    \)
    Therefore, \eqref{upper_endpoint_identity} gives
    \[
        \frac{G'(b)}{G(b)-G(a)}
        =
        \frac{k'(e)}{k(e)}(1+o(1)).
    \]
    Since \(G(a)\to -1\) and \(b\) is bounded below, we obtain
    \[
        \frac{G'(b)}{G(b)-G(a)}
        =
        \frac{(1-G(b))(1+G(b))}{G(b)-G(a)}
        =
        (1-G(b))(1+o(1)),
    \]
    proving \eqref{upper_boundary_balance}.
\end{proof}

The balance \eqref{upper_boundary_balance} identifies the possible limits of the upper boundary, once \(b(\cdot)\) is known to stay bounded from below. The next proposition shows that this lower bound is automatic in the subcritical tail regimes relevant for the balance. The proof has two parts: first we use a simple exit strategy to show that points below the predicted limit eventually belong to the continuation region; then we apply \eqref{upper_boundary_balance} to rule out points above that limit.

For \(0<\lambda<2\), define
\begin{equation}\label{def_beta_lambda}
    \beta_\lambda
    \coloneqq
    G^{-1}(1-\lambda)
    =
    y_0+\frac12\log\frac{2-\lambda}{\lambda}.
\end{equation}
For \(\lambda=0\), we use the convention \(\beta_0=+\infty\).

\begin{Prop}\label{prop_upper_boundary_asymptotics}
    Assume that \(\mathcal I\neq\varnothing\), that
    \(
        k'(s)\to\infty,\, s\to\infty,
    \)
    and that
    \(
        k'(s)/k(s)\to\lambda\in[0,2)
    \)
    as $s\to\infty$. Then
    \begin{equation}\label{upper_boundary_limit_general}
        b(x)\to \beta_\lambda,
        \qquad x\to\infty,\quad x\in\mathcal I.
    \end{equation}
\end{Prop}

\begin{proof}
    First, we prove the lower bound on \(b(x)\). Since \(k'(s)\to\infty\) and
    \(k\) is convex, \(k(s)/s\to\infty\). Also, for fixed \(u,v\in\mathbb R\), the assumption on the logarithmic derivative gives
    \begin{equation}\label{k_u_v}
        \frac{k(x+u)}{k(x+v)}
        =
        \exp\left(
            \int_v^u \frac{k'(x+r)}{k(x+r)}\,dr
        \right)
        \longrightarrow
        e^{\lambda(u-v)}, \quad \text{ as } x \to \infty.
    \end{equation}

    Fix a finite \(y\), and choose \(R>y\). Starting from the displacement
    \(y\), we stop \(Y_y(\cdot)\) at
    \[
        \tau_x
        \coloneqq
        T_{-x}\wedge T_R,
        \qquad
        T_\xi\coloneqq \inf\{t\ge0:Y_y(t)=\xi\}.
    \]
    The scale formula gives
    \begin{equation}\label{prob_right_exit}
        \pr_y(T_R<T_{-x})
        =
        \frac{G(y)-G(-x)}{G(R)-G(-x)}
        \longrightarrow
        \frac{1+G(y)}{1+G(R)}, \quad \text{ as } x \to \infty.
    \end{equation}
    Using the same expected-exit-time ODE calculation as in Step~3 of the proof of Proposition \ref{th_lower_boundary_dichotomy}, now with boundary conditions at \(-x\) and \(R\), we obtain
    \(
        \mathbb E_y[\tau_x]=O(x),
        \, x\to\infty,
    \)
    for fixed \(y<R\).
    Since \(k(x+y)/x\to\infty\) as $x \to \infty$, the running cost is negligible after normalization by \(k(x+y)\). Therefore, from \eqref{k_u_v} and \eqref{prob_right_exit} we obtain
    \begin{equation}\label{upper_boundary_test_ratio}
        \frac{
            \ex_y\big[k(x+Y_y(\tau_x))+c\tau_x\big]
        }
        {k(x+y)}
        \longrightarrow
        e^{\lambda(R-y)}
        \frac{1+G(y)}{1+G(R)}.
    \end{equation}

    Suppose first that \(0<\lambda<2\). Simple algebra shows that the function
    \(
        R\mapsto e^{\lambda R}(1+G(R))^{-1}
    \)
    is minimized at the unique point \(R=\beta_\lambda\). Hence, if \(y<\beta_\lambda\), choosing \(R=\beta_\lambda\) in \eqref{upper_boundary_test_ratio} gives a limiting ratio strictly smaller than one. Thus, \(y\in\mathcal C_x\) for all sufficiently large \(x\). Since this holds for every \(y<\beta_\lambda\), we deduce
    \(
        \liminf_{x\to\infty} b(x)\ge \beta_\lambda.
    \)

    We now prove the matching upper bound. Suppose, toward a contradiction, that along a sequence \(x_n\to\infty\),
    \(
        b(x_n)\ge \beta_\lambda+\varepsilon
    \)
    for some \(\varepsilon>0\). In particular, this means that the sequence \(\{b(x_n)\}_{n \ge 1}\) is bounded below, so Proposition \ref{prop_upper_boundary_balance} applies along this sequence. Hence
    \[
        1-G(b(x_n))
        =
        \frac{k'(x_n+b(x_n))}{k(x_n+b(x_n))}(1+o(1))
        \longrightarrow \lambda.
    \]
    If \(b(x_n)\to+\infty\) along a subsequence, then the left-hand side tends to zero, leading to a contradiction, since we work with $\lambda > 0$. Otherwise, after passing to a subsequence, \(b(x_n)\to L<\infty\), with \(L\ge\beta_\lambda+\varepsilon\), and
    \(
        1-G(L)<1-G(\beta_\lambda)=\lambda,
    \)
    leading again to a contradiction. Therefore
    \(
        \limsup_{x\to\infty} b(x)\le \beta_\lambda,
    \)
    and, together with the lower bound, this proves \(b(x)\to\beta_\lambda\) when \(0<\lambda<2\).

    It remains to consider \(\lambda=0\). In this case, for every finite \(y\), choose any \(R>y\). The limiting ratio in \eqref{upper_boundary_test_ratio} becomes
    \[
        \frac{1+G(y)}{1+G(R)}<1.
    \]
    Hence, every finite \(y\) belongs to \(\mathcal C_x\) for all sufficiently large \(x\). Therefore \(b(x)\to+\infty=\beta_0\).
\end{proof}

The preceding proposition gives explicit asymptotics for the main examples.

\begin{Cor}\label{cor_upper_boundary_examples}
    Assume that the hypotheses of Proposition
    \ref{prop_upper_boundary_asymptotics} hold.

    \begin{enumerate}[label=\textup{(\roman*)}]
        \item For power costs
        \(
            k(s)=\alpha s^\beta+\gamma,
            \, \alpha>0, \beta>1,
        \)
        one has
        \(
            b(x)\to+\infty,
        \)
        and more precisely
        \begin{equation}\label{power_upper_boundary_asymptotic}
            b(x)
            =
            y_0+\frac12\log\frac{2x}{\beta}+o(1), \qquad x \to \infty.
        \end{equation}

        \item For exponential costs
        \(
            k(s)=\alpha(e^{\lambda s}-\gamma),
            \, \alpha>0, 0<\lambda<2,
        \)
        one has
        \begin{equation}\label{exponential_upper_boundary_limit}
            b(x)\to
            y_0+\frac12\log\frac{2-\lambda}{\lambda}.
        \end{equation}
    \end{enumerate}
\end{Cor}

\begin{proof}
    For power costs, \(k'(s)/k(s)\sim \beta/s\), and hence
    \(k'(s)/k(s)\to0\) as $s \to \infty$. Proposition
    \ref{prop_upper_boundary_asymptotics} then gives \(b(x)\to+\infty\) as $x \to \infty$. To obtain the exact asymptotic behavior \eqref{power_upper_boundary_asymptotic}, we apply the
    balance \eqref{upper_boundary_balance} to get
    \[
        1-G(b(x))
        \sim
        \frac{\beta}{x+b(x)}.
    \]
    Since \(b(x)\to+\infty\) as $x \to \infty$, we also have
    \(
        1-G(b(x))
        \sim
        2e^{-2(b(x)-y_0)},
    \)
    as $x \to \infty$.
    Therefore,
    \(
        2e^{-2(b(x)-y_0)}
        \sim
        \beta/(x+b(x)),
    \)
    and taking logarithms gives
    \[
        b(x)
        =
        y_0+\frac12\log\frac{2(x+b(x))}{\beta}+o(1), \quad \text{ as } x \to \infty.
    \]
    In particular, \(b(x)=O(\log(x+b(x)))\), and hence \(b(x)=o(x)\) as $x \to \infty$, which yields
    \eqref{power_upper_boundary_asymptotic}.

    For exponential costs, \(k'(s)/k(s)\to\lambda\). Thus, \eqref{exponential_upper_boundary_limit} follows immediately from \eqref{upper_boundary_limit_general}.
\end{proof}

\subsection{Properties of the value function}
\label{subsec:value_function_properties}

We close this section with the corresponding properties of the value function. These statements are mostly consequences of the fixed-\(x\) solution and of the geometric results above, but it is useful to collect them before returning to the original one-dimensional problem.

First, the fixed-\(x\) representation gives the following global formula.

\begin{Prop}\label{prop_value_representation_2d}
    If \(x\notin\mathcal I\), then
    \(
        U(x,y)=k(x+y)
    \)
    for all $y \ge -x$.
    If \(x\in\mathcal I\), and with \(Q(x)\) and \(R(x)\)  defined in \eqref{def_q_R_coefficients}, we have 
    \begin{equation}\label{value_formula_2d_global}
        U(x,y)
        =
        \begin{cases}
            k(x+y), & -x\le y\le a(x),\\[0.3em]
            Q(x)G(y)+R(x)-cyG(y), & a(x)<y<b(x),\\[0.3em]
            k(x+y), & y\ge b(x).
        \end{cases}
    \end{equation}
\end{Prop}

\begin{proof}
    This is Theorem~\ref{th_fixed_x_solution} applied at the fixed value of
    \(x\), and expressed in the two-dimensional notation of this section.
\end{proof}

This representation also gives smooth fit in the displacement variable. At regular points of the free boundary, it gives smooth fit in the \(x\)-direction as well.

\begin{Prop}\label{prop_smooth_fit_2d}
    Let \(x\in\mathcal I\). Then \(U(x,\cdot)\) satisfies smooth fit at the
    upper boundary:
    \[
        \lim_{y\uparrow b(x)} D_y U(x,y)
        =
        \lim_{y\downarrow b(x)} D_y U(x,y)
        =
        k'(x+b(x)).
    \]
    If \(a(x)>-x\), then \(U(x,\cdot)\) also satisfies smooth fit at the lower
    boundary:
    \[
        \lim_{y\uparrow a(x)} D_y U(x,y)
        =
        \lim_{y\downarrow a(x)} D_y U(x,y)
        =
        k'(x+a(x)).
    \]

    Suppose, in addition, that the boundaries are \(C^1\) on an open interval \(J\subset\mathcal I\). Then \(U\) satisfies horizontal smooth fit as well. More precisely, if \(a(x)>-x\) on \(J\), then
    \begin{equation}\label{x_smooth_fit_interior}
        Q'(x)G(a(x))+R'(x)=k'(x+a(x)),
        \qquad
        Q'(x)G(b(x))+R'(x)=k'(x+b(x)).
    \end{equation}
    If \(a(x)=-x\) on \(J\), then we have the upper-boundary identity
    \[
        Q'(x)G(b(x))+R'(x)=k'(x+b(x)), \qquad x \in J.
    \]
\end{Prop}

\begin{proof}
    The vertical smooth-fit identities are part of the fixed-\(x\) construction in Theorem~\ref{th_fixed_x_solution}. If \(a(x)=-x\), no lower smooth fit is imposed, since the lower endpoint is the boundary of the state space.

    We prove the horizontal identities in the interior-boundary regime; the proof at the upper boundary in the diagonal-contact regime is identical. On \(J\), value matching at the lower boundary gives
    \[
        Q(x)G(a(x))+R(x)-c\,a(x)G(a(x))=k(x+a(x)).
    \]
    Differentiating this identity with respect to \(x\), which we can do by the differentiability of $a, b, G$, $\Phi$, and hence also $Q$ and $R$ in the interior points, we obtain
    \[
        Q'(x)G(a(x))+R'(x)
        +
        \Big[Q(x)G'(a(x))-cG(a(x))-c\,a(x)G'(a(x))\Big]a'(x)
        =
        k'(x+a(x))(1+a'(x)).
    \]
    The expression in the square parentheses is the \(y\)-derivative of the continuation value at \(a(x)\), given in \eqref{value_formula_2d_global}, and by vertical smooth fit it equals \(k'(x+a(x))\). The terms containing \(a'(x)\) therefore cancel, yielding the first identity in \eqref{x_smooth_fit_interior}. The proof at \(b(x)\) is the same, using value matching and vertical smooth fit at the upper boundary.
\end{proof}

Finally, we record three monotonicity properties of \(U\). The first two concern monotonicity in each coordinate separately on the corresponding sections of the natural state space \(\mathcal D\); the third is the value-function counterpart of the diagonal expansion of the continuation region.

\begin{Prop}\label{prop_value_monotonicity_2d}
    For each fixed \(x\ge0\), the function
    \(
        y\mapsto U(x,y),\, y\ge -x
    \)
    is non-decreasing. For each fixed \(y \in \mathbb{R}\), the function
    \(
        x\mapsto U(x,y),\, x \ge \max\{0,-y\}
    \)
    is non-decreasing as well. Moreover, for every \((x,y)\in\mathcal D\) and every \(h\ge0\), we have
    \(
        U(x+h,y-h)\le U(x,y).
    \)
\end{Prop}

\begin{proof}
    We first prove monotonicity in \(y\). Fix \(-x\le y_1\le y_2\), and couple \(Y_{y_1}(\cdot)\) and \(Y_{y_2}(\cdot)\) with the same Brownian motion. Since \(G\) is increasing, the comparison principle gives
    \(
        Y_{y_1}(t)\le Y_{y_2}(t),\, t\ge0.
    \)
    As in the fixed-\(x\) reduction, it is enough to consider stopping times for the process started from \(y_2\) before it hits \(-x\). Given such a stopping time \(\tau\), stop the process started from \(y_1\) at
    \[
        \sigma
        \coloneqq
        \tau\wedge\inf\{t\ge0:Y_{y_1}(t)\le -x\}.
    \]
    Then \(\sigma\le\tau\), and either \(x+Y_{y_1}(\sigma)=0\), or \(\sigma=\tau\) and \(Y_{y_1}(\tau)\le Y_{y_2}(\tau)\). Since \(k\) is non-decreasing on \([0,\infty)\), the cost obtained from \(y_1\) by \(\sigma\) is no larger than the cost obtained from \(y_2\) by \(\tau\). Taking the infimum over \(\tau\) gives
    \(
        U(x,y_1)\le U(x,y_2).
    \)

    The monotonicity in \(x\) follows by the same truncation argument, using the same process \(Y_y(\cdot)\) and truncating at its first hitting time of \(-x_1\), for \(\max\{0,-y\}\leq x_1\leq x_2\).

    The diagonal monotonicity is obtained in the same way as the diagonal expansion in Proposition \ref{prop_diagonal_expansion}, but without taking a strict continuation point. Namely, couple \(Y_y(\cdot)\) with \(\widetilde Y(\cdot)\coloneqq Y_{y-h}(\cdot)\), driven by the same Brownian motion, and set \(Z(t)=\widetilde Y(t)+h\). Then
    \(
        dZ(t)=G(Z(t)-h)\,dt+dN(t),
    \)
    so \(Z(t)\le Y_y(t)\) by comparison. Given any admissible stopping time for \((x,y)\), truncate it as in the proof of Proposition \ref{prop_diagonal_expansion} when \(\widetilde Y\) hits \(-(x+h)\). The resulting stopping time for \((x+h,y-h)\) has no larger running cost and no larger terminal cost. Taking infima over the original stopping times gives
    \(
        U(x+h,y-h)\le U(x,y),
    \)
    as claimed.
\end{proof}

The original value function $V$ in \eqref{value_function_def} is the restriction of \(U\) in \eqref{2d_value_func_def} to the horizontal line \(y=0\). The next section uses the strip representation of \(\mathcal C\), the expansion properties above, and the diagonal-contact dichotomy to describe this restriction in more detail.

\section{Back to the Original Problem}\label{sec:original_problem}

We return now to the original stopping problem of Section~\ref{sec:model}. No new reduction is needed at this point: the original value function in \eqref{value_func_via_Y} is simply the restriction of the two-dimensional value function \eqref{2d_value_func_def} to the horizontal line
\(y=0\), namely, 
\(
    V(x)=U(x,0),\, x\ge0.
\)
The purpose of this short section is therefore to translate the structural results obtained above into the notation of the original problem, whose state is the observed process \(X\) in \eqref{diffusion_process_def}.

The original continuation and stopping sets are
\begin{equation}\label{def_original_continuation_set}
    \mathcal C^0
    \coloneqq
    \{x\ge0:V(x)<k(x)\},
    \qquad
    \mathcal S^0
    \coloneqq
    \{x\ge0:V(x)=k(x)\}.
\end{equation}
Since \(V(x)=U(x,0)\), the original continuation set is the horizontal section of the two-dimensional continuation region:
$x\in\mathcal C^0$ if and only if $(x,0)\in\mathcal C$.
Using the strip representation of \(\mathcal C\), this is equivalent to
\begin{equation}\label{horizontal_section_condition}
    x\in\mathcal C^0
    \quad\Longleftrightarrow\quad
    x\in\mathcal I
    \quad\text{and}\quad
    a(x)<0<b(x).
\end{equation}
It is convenient to encode the two inequalities in
\eqref{horizontal_section_condition} by
\begin{equation}\label{def_kappa_threshold}
    \kappa(x)
    \coloneqq
    \min\{-a(x),\,b(x)\},
    \qquad x\in\mathcal I.
\end{equation}
Thus, the horizontal line \(y=0\) lies inside the continuation strip precisely when \(\kappa(x)>0\).
Finally, we call a point \((x,0)\in\partial\mathcal C\) a regular horizontal
crossing point if \(x\) is an interior point of \(\mathcal I\) and
the free-boundary functions relevant to the corresponding boundary
regime are of class \(C^1\) on a neighborhood of \(x\).

\begin{Th}\label{th_original_threshold}
    Assume that the hypotheses of Theorem~\ref{th_2d_structure} hold. Then
    \begin{equation}\label{original_continuation_via_chi}
        \mathcal C^0
        =
        \{x\in\mathcal I:\kappa(x)>0\}.
    \end{equation}
    In particular, for every \(x\notin\mathcal C^0\), immediate stopping is optimal and \(V(x)=k(x)\). For every \(x\in\mathcal C^0\), we have
    \begin{equation}\label{original_value_formula}
        V(x)=Q(x)G(0)+R(x),
    \end{equation}
    with the functions $Q$ and $R$ defined in \eqref{def_q_R_coefficients}, and an optimal stopping time is
    \begin{equation}\label{original_optimal_time_X}
        \tau_x^*
        =
        \inf\{t\ge0:X(t)\notin(x+a(x),\,x+b(x))\}.
    \end{equation}

    If, in addition, \ref{ass_3} holds, then the function $\kappa(\cdot)$ of \eqref{def_kappa_threshold} is non-decreasing on \(\mathcal I\), and \(\mathcal C^0\) is either empty or a right interval. More precisely, with
    \begin{equation}\label{def_x_star_original}
        x^*
        \coloneqq
        \inf\{x\in\mathcal I:\kappa(x)>0\}
    \end{equation}
    and the convention \(\inf\emptyset=\infty\), immediate stopping is optimal for every \(0\le x<x^*\), while \(x\in\mathcal C^0\) holds for every \(x>x^*\).

    Finally, assume that \(J\subset\mathcal C^0\) is an open interval on which the relevant two-dimensional boundaries are of class \(C^1\). Then \(V\) is \(C^1\) on \(J\), with
    \begin{equation}\label{V_derivative_continuation}
        V'(x)=Q'(x)G(0)+R'(x),
        \qquad x\in J.
    \end{equation}
    If \(x^*<\infty\), \(V(x^*)=k(x^*)\), and \((x^*,0)\) is a regular horizontal crossing point, then \(V\) satisfies smooth fit at the original threshold:
    \begin{equation}\label{original_smooth_fit_threshold}
        V'(x^*+)=k'(x^*).
    \end{equation}
    If \(x^*>0\), then \(V'(x^*-)=k'(x^*)\) holds as well, so \(V\) is continuously differentiable at
    \(x^*\).
\end{Th}

\begin{proof}
    The identity \eqref{original_continuation_via_chi} is just \eqref{horizontal_section_condition} rewritten using \eqref{def_kappa_threshold}. If \(x\notin\mathcal C^0\), then \(V(x)=k(x)\) by the definition of the stopping set. If \(x\in\mathcal C^0\), then \(a(x)<0<b(x)\), so \((x,0)\in\mathcal C\).
    The value representation in Proposition \ref{prop_value_representation_2d} gives
    \(
        V(x)=U(x,0)
    \)
    as in \eqref{original_value_formula}, since the term \(-cyG(y)\) vanishes at \(y=0\). The stopping time \eqref{original_optimal_time_X} is just the first exit time from the fixed-\(x\) continuation interval.

    Under \ref{ass_3}, Proposition \ref{prop_vertical_expansion} gives
    \(
        a(x+h)\le a(x)
    \)
    and
    \(
        b(x+h)\ge b(x)
    \)
    for all $h \ge 0$.
    Hence, the functions \(-a(\cdot)\), \(b(\cdot)\), and therefore \(\kappa(\cdot)\), are non-decreasing on \(\mathcal I\). Since \(\mathcal I\) is a right interval by Proposition~\ref{prop_diagonal_expansion}, the set \(\{x\in\mathcal I:\kappa(x)>0\}\) is a right interval. This proves the
    threshold statement.

    On any open interval \(J\subset\mathcal C^0\) where the relevant two-dimensional boundaries are of class \(C^1\), the functions \(Q\) and \(R\) are of class \(C^1\), and differentiating in \eqref{original_value_formula} leads to \eqref{V_derivative_continuation}.

    Finally, if \((x^*,0)\) is a regular horizontal crossing point, then the horizontal smooth-fit statement in Proposition \ref{prop_smooth_fit_2d} gives the continuation-side derivative \(V'(x^*+)=k'(x^*)\). On the stopping side, \(V(x)=k(x)\). Therefore, if \(x^*>0\), \(V'(x^*-)=k'(x^*)\), and \(V\) is of class \(C^1\) at the threshold.
\end{proof}

Theorem \ref{th_original_threshold} identifies the original continuation set as the horizontal section of the two-dimensional continuation strip. However, it does not, by itself, decide whether this section is nonempty, nor, if it is nonempty, whether it is a half-line, a bounded interval, or a union of disjoint intervals. There are two possible obstructions. The first is that the two-dimensional continuation region may itself be empty, that is, \(\mathcal I=\varnothing\). This question is governed by the birth criterion in Proposition \ref{prop_birth_sections}. 
The second obstruction can occur even when \(\mathcal I\neq\varnothing\): the continuation strip may lie entirely below the horizontal line \(y=0\), in which case the original continuation set is still empty. Alternatively, the strip may have a nontrivial intersection with $y=0$, in which case the original problem also has a nontrivial continuation set. All of these possibilities can occur; they are illustrated in Section \ref{sec:examples}.

Regarding the second obstruction, it can be partially clarified by the boundary asymptotics of subsection \ref{subsec:upper_boundary_asymptotics}. Suppose that \(\mathcal I\neq\varnothing\), that \ref{ass_3} holds, and that the upper-boundary asymptotics of Proposition \ref{prop_upper_boundary_asymptotics} apply:
\[
    \frac{k'(s)}{k(s)}\to\lambda\in[0,2),
    \qquad
    b(x)\to\beta_\lambda .
\]
Since \(d(x)=x+a(x)\) is non-negative and non-increasing on \(\mathcal I\), the lower boundary satisfies
\[
    a(x)=d(x)-x<0
\]
for all sufficiently large \(x\in\mathcal I\). Hence, along the tail of \(\mathcal I\), the horizontal line \(y=0\) enters the continuation strip precisely when the upper boundary eventually lies above zero. Under \ref{ass_3}, the boundary \(b(\cdot)\) is non-decreasing, and therefore
\[
    \mathcal C^0\neq\varnothing
    \quad\Longleftrightarrow\quad
    x^*<\infty
    \quad\Longleftrightarrow\quad
    \beta_\lambda>0,
\]
in the notation of \eqref{def_beta_lambda}.
In particular, if \(\lambda=0\), then \(\beta_0=+\infty\), and the original continuation threshold is finite whenever \(\mathcal I\neq\varnothing\).

We conclude by mentioning that the analysis above was carried out for \(x\ge0\). The negative half-line is obtained by symmetry. To make the dependence on the prior explicit, write \(V_p\) for the value function when \(\mathbb P(B=1)=p\). Since \(k\) is even, reflecting the observation process sends \(x\) to \(-x\), \(B\) to \(-B\), and the prior \(p\) to \(1-p\). Hence
\begin{equation}\label{symmetry_negative_x_original}
    V_p(x)=V_{1-p}(-x),
    \qquad x<0.
\end{equation}
Thus, the solution on the positive half-line, together with the replacement \(p\mapsto1-p\), determines the value function and the optimal stopping rule on the whole real line.

The cost-specific consequences of the lower- and upper-boundary asymptotics, including the power and exponential examples, are recorded in the next section.

\section{Examples}\label{sec:examples}

We finish with several examples illustrating the geometry developed in previous sections. The main purpose of this section is thus to show how different structural regimes of the two-dimensional continuation region can appear for concrete costs and parameter values.

As in Section~\ref{sec:structure}, several qualitatively distinct features may occur.  First, the two-dimensional continuation region may be empty or nonempty, and even when it is nonempty, it might not have an intersection with the horizontal line \(y=0\). 
In the latter case, the embedded two-dimensional problem has a continuation strip, but the original one-dimensional problem mandates immediate stopping everywhere.  
Second, the continuation strip may be born already at \(x=0\), or only after a positive value of \(x\).  Third, the lower boundary may or may not touch the diagonal \(y=-x\). Fourth, the upper boundary is monotone under Assumption~\ref{ass_3}, but can behave differently when this assumption is dropped. Finally, the main fixed-\(x\) construction in Section~\ref{sec:one_d_problem} relies heavily on the one-valley geometry of \(\Phi_x\) in \eqref{def_phi}; for more general convex costs this geometry need not hold.

The examples below are chosen to make these alternatives visible. Unless otherwise stated, the shaded regions in the two-dimensional plots represent the continuation region \(\mathcal C\), the curves \(a(\cdot)\) and \(b(\cdot)\) are the lower and upper boundaries, the dashed line is the diagonal \(y=-x\), and the horizontal line is \(y=0\).  The vertical dotted line marks the numerical birth point of the nonempty continuation sections.

\subsection{A standard quadratic example}\label{subsec_quadratic_example}

We start with the quadratic cost
\[
    k(s)=s^2,\qquad c=0.6,\qquad p=0.5;
\]
the corresponding two-dimensional continuation region is illustrated in the left panel of  Figure~\ref{fig:example_quadratic_standard}.
This is the most standard case among the examples considered in this section. The cost satisfies all standing assumptions, including \ref{ass_2} and \ref{ass_3}. The continuation strip is born only after a positive value of \(x\), numerically around \(x\simeq2.03\).  Since \(k'(0+)=0<c\), Proposition~\ref{prop_no_contact_m_le_c} implies that the lower boundary never touches the diagonal.  More precisely, Corollary \ref{cor_lower_boundary_asymptote} gives \(a(x)=-x+\rho+o(1)\), with \(\rho=c/2\) in this example.  On the other hand, the upper boundary increases logarithmically by Corollary~\ref{cor_upper_boundary_examples}. Thus, the horizontal line \(y=0\) eventually enters the continuation strip and remains inside it thereafter.

\begin{figure}[ht]
    \centering
    \begin{subfigure}[t]{.48\linewidth}
        \centering
        \includegraphics[width=\linewidth]{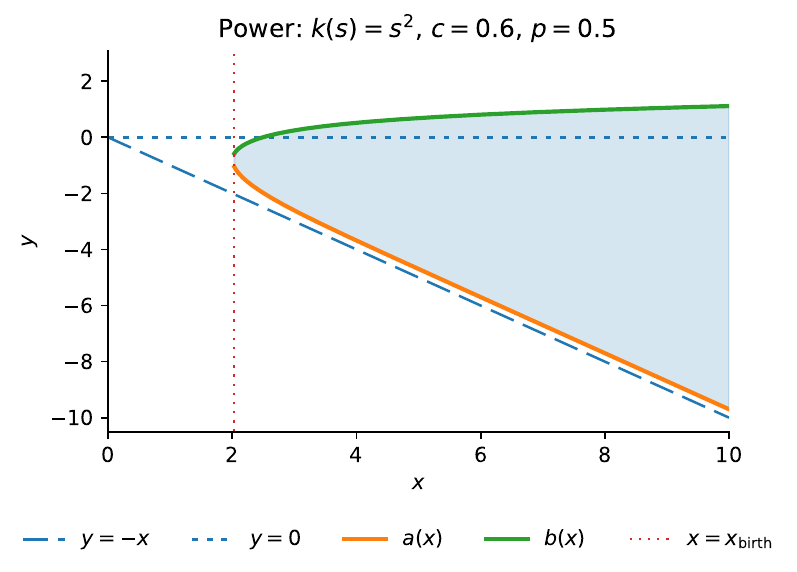}
        \caption{Continuation strip.}
    \end{subfigure}
    \hfill
    \begin{subfigure}[t]{.4\linewidth}
        \centering
        \includegraphics[width=\linewidth]{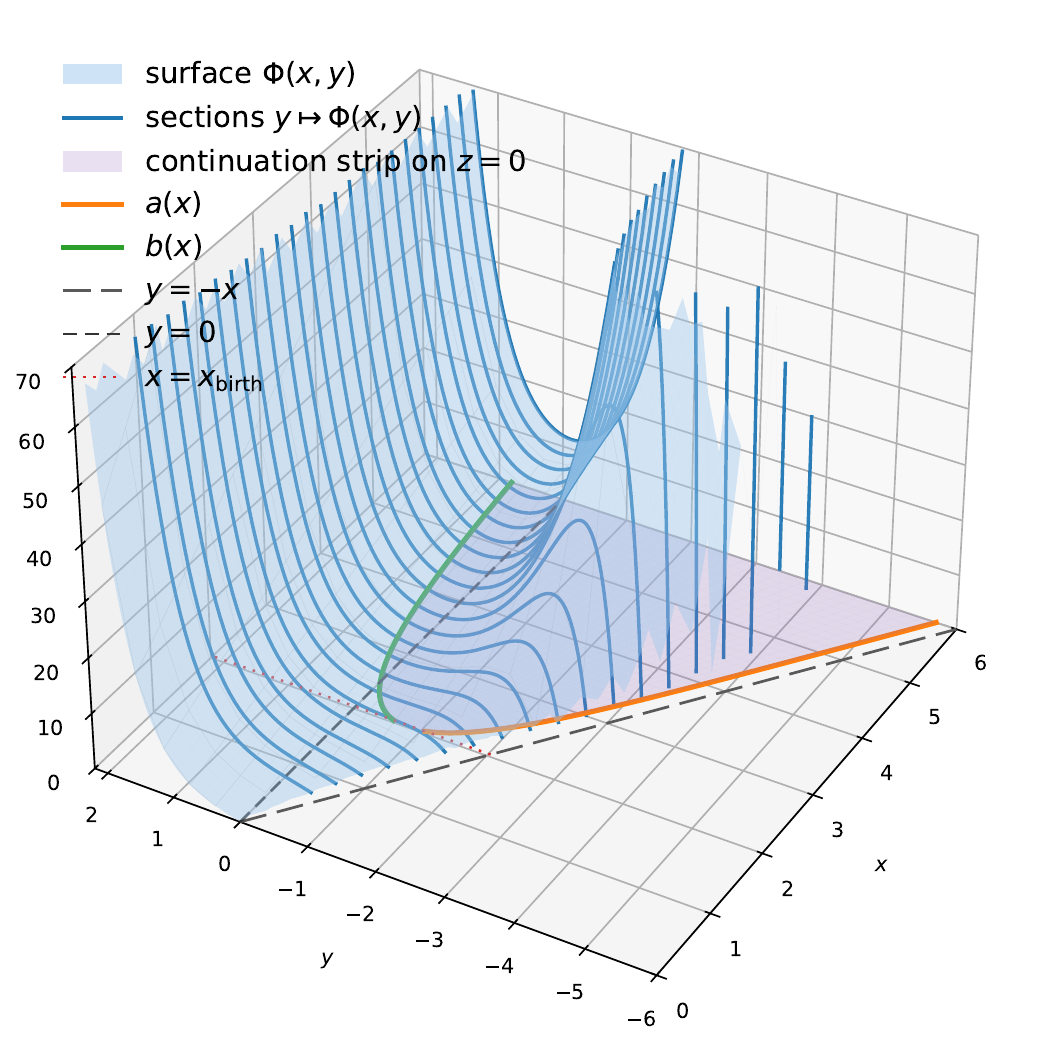}
        \caption{Evolution of \(\Phi_x\).}
    \end{subfigure}
    \caption{The standard quadratic example \(k(s)=s^2\), with \(c=0.6\) and \(p=0.5\).  The left panel shows the continuation strip.  The right panel shows the surface \(\Phi(x,y)\), several fixed-\(x\) sections \(y\mapsto\Phi(x,y)\), and the continuation strip projected onto the plane \(z=0\).} \label{fig:example_quadratic_standard}
\end{figure}

Figure~\ref{fig:example_quadratic_standard} also shows the same example from a different perspective.  The right panel plots the surface \((x,y)\mapsto\Phi(x,y)\), together with several fixed-\(x\) sections.  For small \(x\), the function \(y\mapsto\Phi(x,y)\) is monotone and the fixed-\(x\) continuation section is empty.  Once \(x\) passes the birth point, a single-valley region appears and then expands; the projected strip on the plane
\(z=0\) shows the corresponding continuation region.

\subsection{Birth at the origin and diagonal contact}

The next two examples illustrate different behavior near the birth of the continuation strip.  In the left panel of Figure~\ref{fig:example_power_linear} we again use a power cost
\(
    k(s)=s^2,
\)
this time with $c=0.1,\, p=0.03$.
Here, the prior is tilted sufficiently toward the negative drift so that the continuation region is already nonempty at \(x=0\).  Since \(k'(0+)=0<c\), the lower boundary still does not touch the diagonal.

The right panel shows the linear cost
\(
    k(s)=2s
\)
with $c=0.8$ and $p=0.2.$
Here \(k'(0+)=2>c\), and the lower boundary is pinned to the diagonal as soon as continuation appears.  This is the simplest example of diagonal contact.

\begin{figure}[ht]
    \centering
    \includegraphics[width=.92\linewidth]{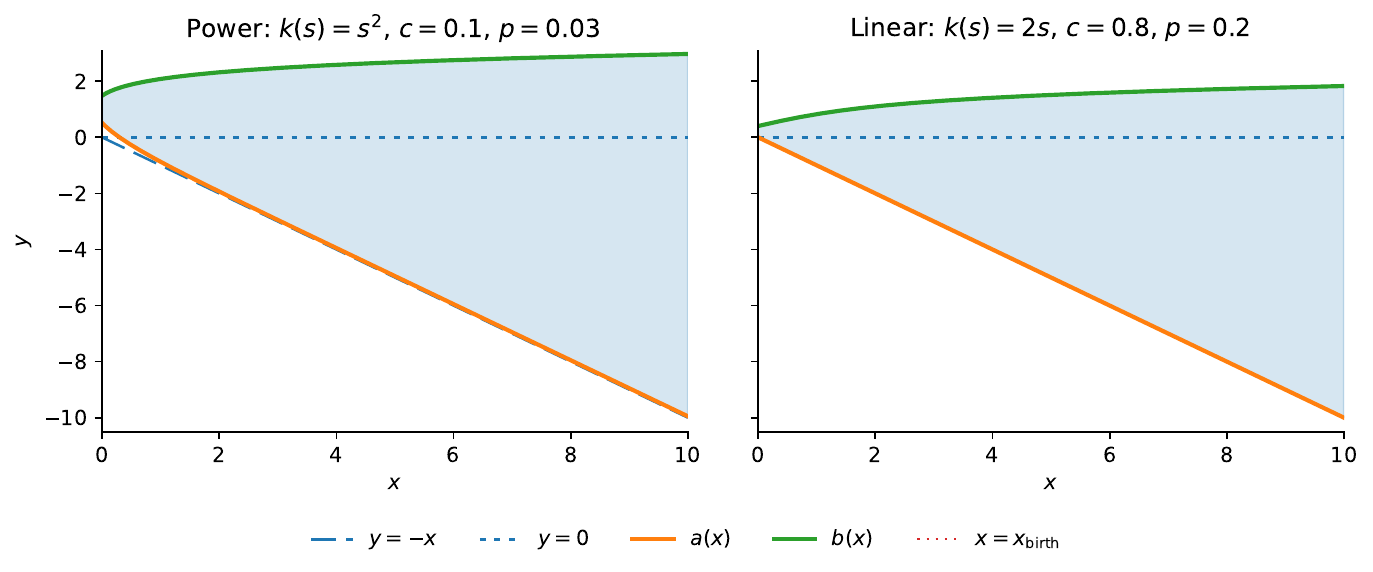}
    \caption{Two-dimensional continuation regions for two costs. The left panel shows quadratic cost $k(s) = s^2$ with $c=0.1$, $p=0.03$. The right panel shows linear cost $k(s)=2s$ with $c=0.8$, $p=0.2.$}
    \label{fig:example_power_linear}
\end{figure}

\subsection{Exponential costs and empty $\mathcal{C}^0$}\label{subsec_exponential_example}

Exponential costs give a different type of upper-boundary behavior. For
\(
    k(s)=\alpha e^{\lambda s}
\)
with \(0<\lambda<2\), Proposition~\ref{prop_upper_boundary_asymptotics} gives
\[
    b(x)\to \beta_\lambda
    =
    y_0+\frac12\log\frac{2-\lambda}{\lambda}.
\]
Thus, the upper boundary converges to a finite level, and the sign of \(\beta_\lambda\) determines whether the horizontal line \(y=0\) eventually enters the strip.

Figure~\ref{fig:example_exponential} shows two such regimes. In the left panel, the cost is given by
\(
    k(s)=e^{1.5s}
\)
with $c=1$ and $p=0.1$. Here \(\beta_\lambda>0\), and the original one-dimensional continuation region is nonempty.  The same example also illustrates delayed diagonal contact of the lower boundary. In the right panel,
\(
    k(s)=0.3e^{1.7s}
\)
with $c=0.8$ and $p=0.3.$
The two-dimensional continuation strip exists, but its upper boundary remains below \(y=0\). Therefore, the original one-dimensional continuation region is
empty even though \(\mathcal I\neq\varnothing\).

\begin{figure}[ht]
    \centering
    \includegraphics[width=.92\linewidth]{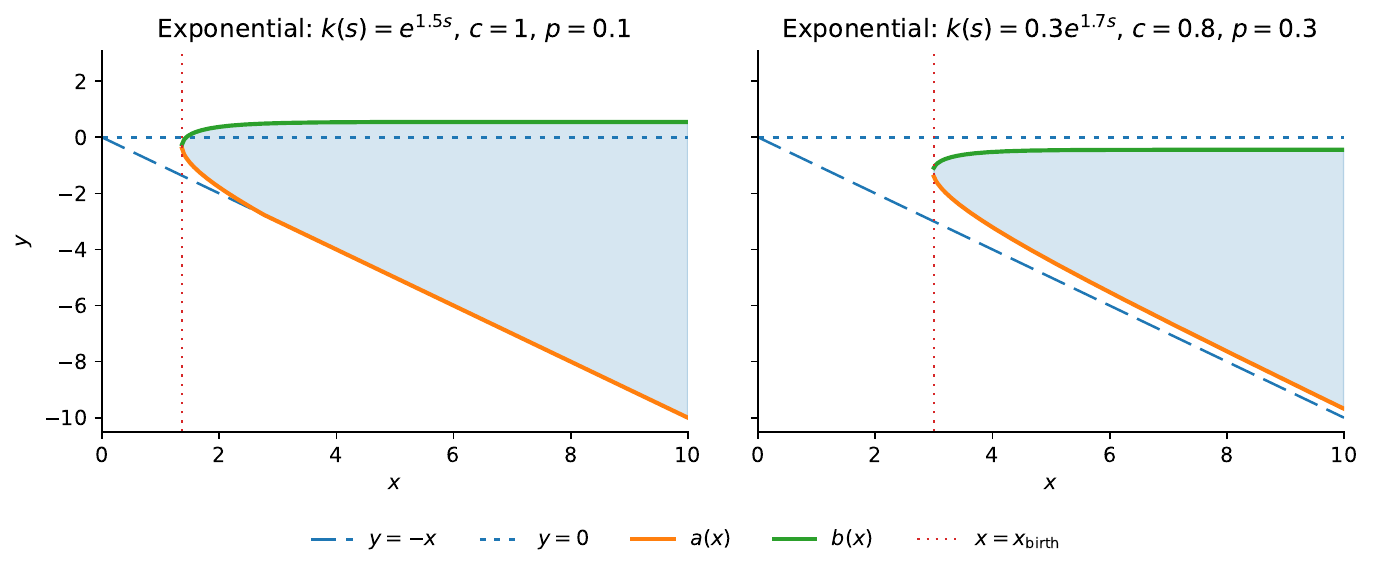}
    \caption{Two exponential examples with different behaviors of the continuation set for the original problem $\mathcal{C}^0 = \mathcal{C} \cap \{y = 0\}$.}
    \label{fig:example_exponential}
\end{figure}

\subsection{Examples outside the monotonicity assumptions}

We now show two numerical examples outside the monotonicity framework of assumption~\ref{ass_3}.  Both costs satisfy the basic standing assumption \ref{ass_1}, but not the structural condition \ref{ass_2}; in particular, the convexity of the function \(P_k\) in \eqref{Qk_def} fails for these choices.  This does not by itself imply that the fixed-\(x\) one-valley geometry fails.  In fact, for the parameter ranges shown below, a direct numerical check of the sign of \(\Phi_x'\) shows a single interval of decrease for each plotted value of \(x\). Thus, the examples should be read as illustrations of what can happen when the vertical monotonicity assumption is dropped, while the fixed-\(x\) sections remain numerically in the one-valley regime.

\begin{figure}[ht]
    \centering
    \includegraphics[width=.92\linewidth]{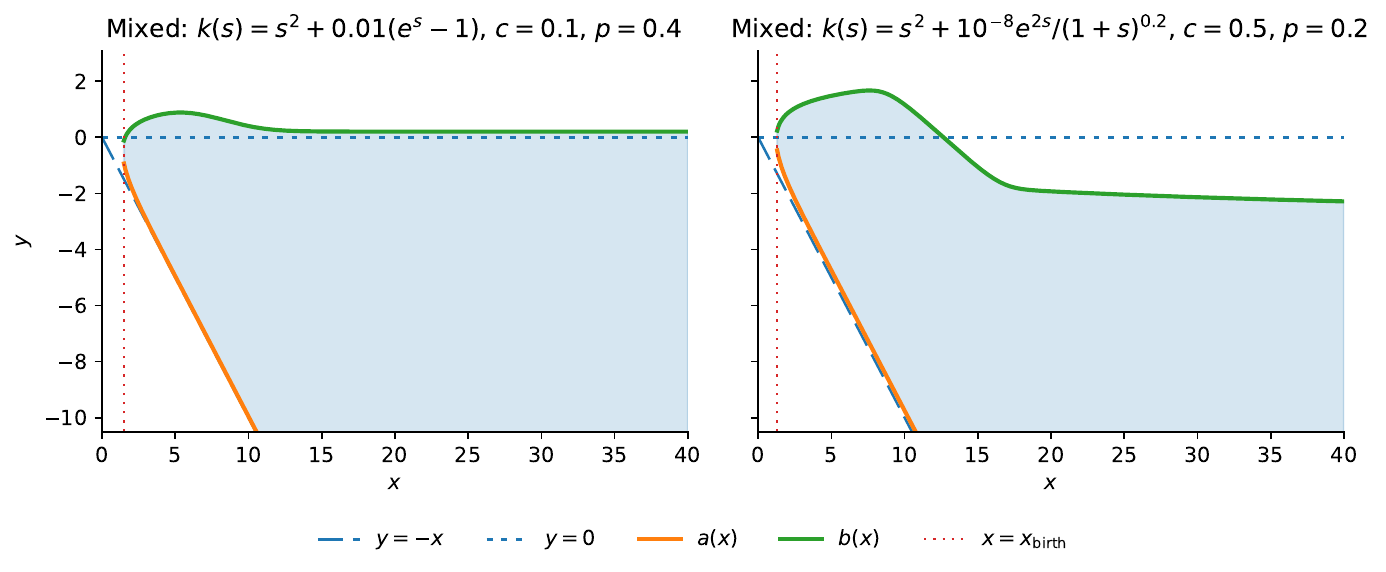}
    \caption{Two examples with mixed power/exponential costs that fail to satisfy assumption \ref{ass_3}, thus also monotonicity of the upper boundary.}
    \label{fig:example_nonmonotone_b}
\end{figure}

The left panel of Figure~\ref{fig:example_nonmonotone_b} uses the mixed cost
\[
    k(s)=s^2+0.01(e^s-1),\qquad c=0.1,\qquad p=0.4.
\]
The exponential tail eventually dominates the power part, and the upper boundary \(b(\cdot)\) bends downward after initially increasing. Nevertheless, the upper boundary remains above the horizontal line \(y=0\), so the original
continuation region is still eventually nonempty.

The right panel uses a critical-tail perturbation,
\[
    k(s)=s^2+10^{-8}\frac{e^{2s}}{(1+s)^{0.2}},
    \qquad c=0.5,\qquad p=0.2.
\]
This cost is still convex, but its logarithmic derivative tends to the critical value \(2\).  Numerically, the upper boundary first rises above zero and then falls below it, eventually drifting downward. Thus, the original one-dimensional continuation region is a finite interval of \(x\)-values, rather than a right half-line.  This behavior is excluded by the monotonicity conclusions obtained under assumption~\ref{ass_3}.

\subsection{Failure of the one-valley geometry}

Finally, we illustrate why the one-valley assumption in the fixed-\(x\) problem of Section \ref{sec:one_d_problem} is a genuine structural restriction. We consider a smooth strictly convex cost $k(\cdot)$ defined on \([0,\infty)\) by prescribing
\[
    k'(s)
    =
    a_0+\mu s+\frac{\mu\delta}{\omega}\sin(\omega s), \quad k(0) = 0,
\]
with
\(
    a_0=0.2,\, \mu=2,\, \delta=0.9,\, \omega=12.
\)
Then \(k'(s)>0\), and
\(    
    k''(s)
    =
    \mu\big(1+\delta\cos(\omega s)\big)
    \ge
    \mu(1-\delta)>0.
\)
Thus, \(k(\cdot)\) satisfies the basic convexity and regularity assumptions of \ref{ass_1} after even extension. However, the oscillatory curvature violates the simple sufficient condition \ref{ass_2}, and the resulting geometry of the critical function \(\Phi_x\) in \eqref{def_phi} is no longer one-valley.
Indeed, for the fixed value
\[
    x=1.5,\qquad c=0.2,\qquad p=0.5,
\]
the function \(\Phi_x\) has several distinct intervals of decrease. This is shown in the left panel of Figure~\ref{fig:example_multivalley}.  The orange regions indicate where \(\Phi_x'<0\), while the pale blue regions indicate the continuation intervals obtained numerically for this fixed \(x\).

The middle panel plots the numerical gap \(k_x-U_x\), computed from a finite-difference solution of the fixed-\(x\) obstacle problem.  The gap is positive on several disjoint intervals, confirming the multi-component structure of the continuation set in this example.  The right panel shows the corresponding numerical two-dimensional continuation region.  The vertical dotted line marks the value \(x=1.5\) used in the first two panels; its intersection with the two-dimensional region gives the fixed-\(x\) continuation intervals shown on the left and in the middle.

This example is included to show that the one-interval geometry imposed in Section \ref{sec:one_d_problem} is not automatic, even for smooth strictly convex terminal costs.  Once this geometry is dropped, the fixed-\(x\) continuation set, and consequently the two-dimensional continuation region, may have a substantially more complicated structure.

\begin{figure}[ht]
    \centering
    \includegraphics[width=.98\linewidth]{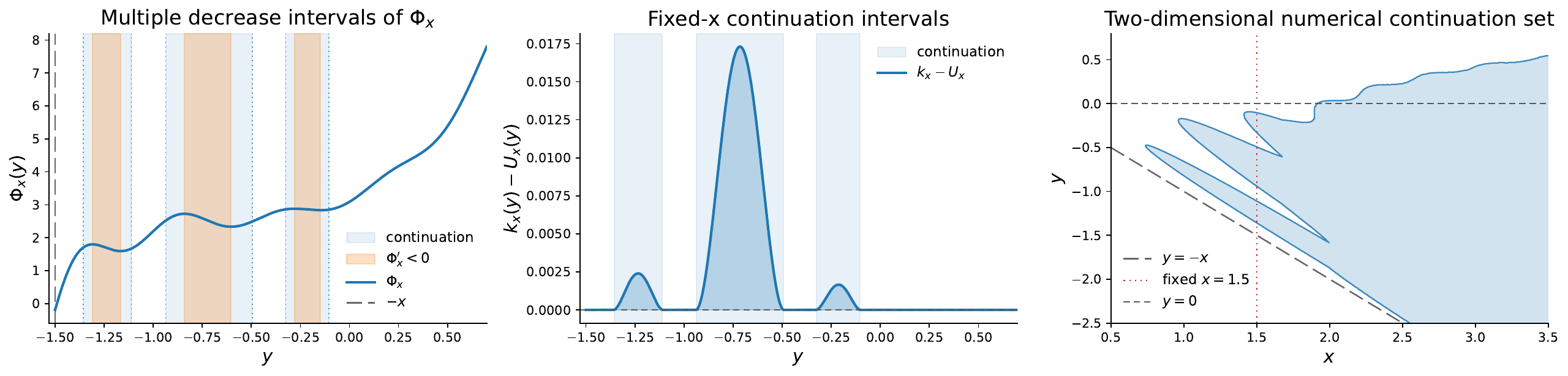}
    \caption{A numerical example outside the one-valley geometry.}
    \label{fig:example_multivalley}
\end{figure}

\section{Open Problems}\label{sec:conclusion}

Several questions, both analytic and geometric, remain open. We believe that there are many possible refinements of our results, but five seem particularly interesting.

The first concerns the geometry of the free boundaries $a(\cdot)$ and $b(\cdot)$. In the examples of Subsections \ref{subsec_quadratic_example}--\ref{subsec_exponential_example}, all of which are covered by the structural assumptions \(\ref{ass_1}\)--\(\ref{ass_3}\), the lower boundary \(a(\cdot)\) appears to be convex, while the upper boundary \(b(\cdot)\) appears to be concave.  We do not know whether this is a general phenomenon, nor what assumptions on \(k\), or equivalently on the functions \(\Phi_x\), might imply it. It would be very interesting to establish such curvature properties, since they would give a much sharper geometric description of the continuation strip.

The second question concerns the regularity and characterization of the free boundaries. We have shown that the boundaries are described by balancing equations \eqref{eq_main_balancing_condition}, and at regular points by a coupled system of ordinary differential equations \eqref{ode_a_interior} and \eqref{ode_b_active}.
It would be useful to identify intrinsic conditions under which the corresponding regularity assumptions can be removed, for instance conditions guaranteeing global \(C^1\)-regularity of the boundaries. More generally, one would like a characterization of the boundary pair \((a,b)\) which is less dependent on the nondegeneracy assumptions appearing in the local ODE description.

The third question is to obtain a more explicit analytic characterization of the birth of continuation in the original problem. For the two-dimensional problem, the birth of nonempty fixed-\(x\) sections is characterized by the condition that \(D_y \Phi(x,y)\) becomes negative; see \eqref{def_I_via_phi}. However, the original one-dimensional continuation region is the horizontal section \(y=0\) of the two-dimensional continuation strip. Thus, its birth depends on when this horizontal line first enters the strip, and we do not have a comparable analytic formula for this threshold.

The fourth question is to replace the Bernoulli prior by a more general prior on the unknown drift. The Bernoulli case is special because, after filtering and centering by the initial position, the problem remains a time-homogeneous optimal stopping problem. For a general prior, the posterior distribution may itself become part of the state, and the resulting stopping problem is expected to be time-inhomogeneous in general. It would be interesting to understand which of the qualitative features found here survive in that setting. In particular, one may ask under what conditions the continuation region remains connected, or even convex, in an appropriate state space. As already mentioned, this is one of the reasons for using the filtering formulation throughout the paper, rather than relying on the more classical change-of-measure approach often used for two-point priors in sequential testing: the general filtering approach seems more adaptable to priors beyond the Bernoulli case.

The final question, already mentioned in the introduction and the one we find the most intriguing, is to combine stopping with control of information acquisition leading to a triple problem of stochastic control, stopping, and filtering. To be more precise, in the present paper the observation mechanism is fixed. Instead, one could allow the decision maker to control the intensity or precision of the signal, paying a running cost of control. This would lead to a joint optimal stopping and stochastic control problem: one has to decide not only when to stop, but also how much information to acquire before stopping. It would be natural to ask whether the geometric viewpoint developed here still leads to tractable free-boundary equations, and whether the optimal information-acquisition policy has a simple structure.

\vspace{15pt}
\noindent
\textbf{Acknowledgments}

We gratefully acknowledge financial support from the National Science Foundation under grant NSF-DMS-25-06199 and from a Lenfest Award at Columbia University.

\printbibliography

\end{document}